\documentclass[a4paper,11pt, twoside]{article}
\usepackage{amsmath,amsthm,esint}
\usepackage{amssymb,latexsym}
\usepackage{mathrsfs}
\usepackage{enumerate}
\usepackage{graphicx}
\usepackage[colorlinks,citecolor=red]{hyperref}
\usepackage[numbers,sort&compress]{natbib}
\usepackage{indentfirst}
\usepackage{mathtools}
\usepackage{paralist,bbding,pifont}
\newtheorem{theorem}{Theorem}[section]

\newtheorem{lemma}{Lemma}[section]
\newtheorem{proposition}{Proposition}[section]
\newtheorem{definition}{Definition}[section]
\newtheorem{remark}{Remark}[section]

\theoremstyle{definition} \theoremstyle{remark}
\numberwithin{equation}{section}
\allowdisplaybreaks

\begin{document}

\markboth{Y.Z. Yang, Y. Zhou et al}{space-time fractional nonlocal operators}

\date{}

\baselineskip 0.22in

\title{\bf Muckenhoupt-weighted $L_q(L_p)$ boundedness for time-space fractional nonlocal operators}

\author{ Yong Zhen Yang$^1$, Yong Zhou$^{1,2}$\\[1.8mm]
\footnotesize {Correspondence: yozhou@must.edu.mo}\\
\footnotesize  {$^{1}$ Faculty of Mathematics and Computational Science, Xiangtan University}\\
\footnotesize  {Hunan 411105, P.R. China}\\[1.5mm]
\footnotesize {$^2$ Macao Centre for Mathematical Sciences, Macau University of Science and Technology}\\
\footnotesize {Macau 999078, P.R. China}\\[1.5mm]
}

\maketitle

\begin{abstract}
We develop a weighted mixed-norm $L_q(L_p)$-estimates for solutions to fractional evolution equations of the form
\[
\partial_t^\alpha w(t,x) = \phi(\Delta) w(t,x) + h(t,x), \quad w(0,\cdot) = w_0, \quad t > 0, \; x \in \mathbb{R}^d,
\]
where $\partial_t^\alpha$ denotes the Caputo derivative of $\alpha \in (0,1)$ and $\phi(\Delta)$ is a nonlocal operator associated with a Bernstein function $\phi$. For all $p, q \in (1, \infty)$ and $\gamma \in \mathbb{R}$, we prove the estimate
\begin{align*}
&\left\| \partial_t^\alpha w \right\|_{L_q(0,T,\mu_2dt; H^{\phi,\gamma}_p(\mu_1))} + \left\| \phi(\Delta) w \right\|_{L_q(0,T,\mu_2dt; H^{\phi,\gamma}_p(\mu_1))} \\
&\qquad\leq C \left( \left\| h \right\|_{L_q(0,T,\mu_2dt; H^{\phi,\gamma}_p(\mu_1))} + \left\| w_0 \right\|_{N_{\alpha,p,\phi}} \right),
\end{align*}
where $\mu_1\in A_p(\mathbb{R}^d)$ and $\mu_2\in A_q(\mathbb{R})$ are Muckenhoupt weights, and $N_{\alpha,p,\phi}$ is a Banach space characterizing admissible initial data.
In particular, when $\mu_2\equiv 1$ and $\alpha q>1$, $N_{\alpha,p,\phi}$ coincides with the weighted Besov space $B^{\phi,\gamma+2-\frac{2}{\alpha q}}_{p,q}(\mu_1)$.

The analysis employs tools from harmonic analysis, including the Fefferman--Stein inequality, Hardy-Littlewood maximal estimates in weighted mixed-norm spaces, and sharp function methods for bounding solution operators. These results extend and unify previous work by K.~H.~Kim et al, providing a general analytic framework for weighted \(L_q(L_p)\)-theory of time-space nonlocal evolution equations.\\ [2mm]
\end{abstract}

\baselineskip 0.25in

\section{Introduction}
The mathematical framework of fractional calculus has emerged as an indispensable tool for modeling complex systems across diverse scientific domains. Unlike classical calculus, fractional operators inherently capture non-local effects and memory phenomena, making them particularly suitable for describing viscoelastic materials, anomalous transport processes, and hereditary properties in physical systems. This modeling approach provides distinct advantages in several interconnected areas of applied mathematics and physics. In continuum mechanics, fractional derivatives enable accurate characterization of viscoelastic behavior and fractal media flow. Electromagnetic theory benefits from their application to wave propagation analysis and non-Ohmic conduction problems. Biological systems modeling utilizes fractional calculus for describing neuronal signaling dynamics and bioelectrical impedance, while control theory employs it in the design of robust fractional-order controllers. The methodology's effectiveness stems from its ability to represent complex dynamics with physically interpretable parameters, often yielding more parsimonious models than conventional approaches. Comprehensive mathematical treatments of these applications are available in \cite{A.A.Kilbas,Y zhou2,Podlubny}.

In this paper, we establish the following weighted $L_{q}(L_{p})$ theory with Muckenhoupt weights for time-space fractional evolution equations (TSFEs) driven by $\phi(\Delta)$-type operators:
\begin{align}\label{TSFE}
\begin{cases}
\partial_t^\alpha w(t,x) = \phi(\Delta)w(t,x) + h(t,x), & t>0,\ x\in\mathbb{R}^{d}, \\
w(0,\cdot) = w_{0}, & x\in\mathbb{R}^{d},
\end{cases}
\end{align}
where  $w_{0}$ is the initial value, $\partial_{t}^{\alpha}$ denotes the Caputo derivative of order $\alpha\in (0,1)$, $h(t,x)$ is a given function, and $\phi$ denotes a Bernstein function with $\phi(0^{+})=0$, mapping $(0,\infty)$ to $(0,\infty)$ that satisfies
\[
(-1)^{k}\phi^{(k+1)}(x) \geq 0, \quad x > 0, \quad k = 0, 1, 2, \ldots
\]
The operator $\phi(\Delta):=-\phi(-\Delta)$ represents the quantization of rotationally invariant subordinate Brownian motion with characteristic exponent $\phi(|\xi|^{2})$, defined as
\[
\phi(\Delta)w(x) = \mathcal{F}^{-1}\left(-\phi(|\xi|^{2})\mathcal{F}w(\xi)\right)(x), \quad w \in \mathcal{S}(\mathbb{R}^{d}).
\]

When $\phi(x)=x^{\frac{\beta}{2}}$ ($0<\beta<2$), the operator $\phi(-\Delta)$ reduces to the fractional Laplacian $(-\Delta)^{\frac{\beta}{2}}$. For $\phi(\Delta)=\Delta$, TSFEs \eqref{TSFE} become the fractional diffusion equation:
\begin{align}\label{TFE}
\begin{cases}
\partial_t^\alpha w(t,x) =\Delta w(t,x) + h(t,x), & t>0,\ x\in\mathbb{R}^{d}, \\
w(0,\cdot) = w_{0}, & x\in\mathbb{R}^{d}.
\end{cases}
\end{align}
Particle sticking and trapping phenomena in subdiffusive systems are mathematically represented by this equation \cite{Metzler2,Y zhou2}.

Han \cite{Han} investigated the unique solvability of \eqref{TFE} when $w_{0}\equiv 0$ in weighted $L_{q}(L_{p})$ spaces by exploiting the scaling properties of solution operators combined with harmonic analysis techniques. However, for TSFEs \eqref{TSFE}, the lack of scaling properties in the symbol $\phi(|\xi|^{2})$ renders Han's approach inapplicable. In this work, we overcome this limitation and establish the unique solvability of TSFEs \eqref{TSFE} with non zero initial value $w_{0}$ in weighted $L_{q}(L_{p})$ spaces, obtaining the following estimate:
$$
\left\|\left|\partial_t^\alpha w\right|+\left|w\right|+\left|\phi(\Delta)w\right|\right\|_{\mathcal{H}^{\phi,\gamma}_{p,q}(\mu_{1},\mu_{2},T)}\leq C\left\|h\right\|_{\mathcal{H}^{\phi,\gamma}_{p,q}(\mu_{1},\mu_{2},T)}+\big\|w_{0}\big\|_{N_{\alpha,p,\phi}},
$$
where $1<p,q<\infty$, $\mu_{1}(x)\in A_{p}(\mathbb{R}^{d})$, $\mu_{2}(t)\in A_{p}(\mathbb{R})$, and the norm $\mathcal{H}^{\phi,\gamma}_{p,q}(\mu_{1},\mu_{2},T)$ is defined by
$$
\left\|h\right\|_{\mathcal{H}^{\phi,\gamma}_{p,q}(\mu_{1},\mu_{2},T)}=\left(\int_{0}^{T}\left\|h(t)\right\|^{q}
_{H^{\phi,\gamma}_{p}(\mu_{1})}\mu_{2}(t)\,dt\right)^{\frac{1}{q}}.
$$

The $L_{q}(L_{p})$ theory for time-fractional diffusion equations has been an active research area. Below we summarize relevant works and contributions. In \cite{Clement}, using semigroup theory under the conditions
\begin{align*}
\alpha\in (0,1),\quad a_{ij}(t,x)=\delta_{ij},\quad \frac{2}{\alpha q}+\frac{d}{p}<1,
\end{align*}
the authors studied the unweighted $L_{q}(L_{p})$ theory for the following Volterra equation:
\begin{align}\label{volterra equation}
\partial^{\alpha}_{t}w(t,x)=a_{ij}(t,x)w_{x^{i}x^{j}}(t,x)+h(t,x).
\end{align}
In \cite{Kim3}, by employing Calder\'{o}n-Zygmund theory, the restrictions on $a_{ij}(t,x)$ were relaxed to allow uniform continuity in $x$ and piecewise continuity in $t$, establishing the $L_{q}(L_{p})$ theory for \eqref{volterra equation} when $\alpha\in(0,2)$ and $p,q>1$. In \cite{Dong 1}, using the level set method and the "crawling of ink spots" lemma (\cite[Appendix]{Dong 1}), the authors investigated the unweighted $L_{p}$ theory for $\alpha\in (0,1)$ with $p=q$, while relaxing the conditions on $a_{ij}(t,x)$ to only require measurability in $t$ and allowing small oscillations in $x$. However, these iterative techniques and level set methods cannot be naturally extended to mixed-norm $L_{q}(L_{p})$ spaces or incorporate weights. In \cite{Han}, the authors studied the weighted $L_{q}(L_{p})$ estimates for \eqref{volterra equation} under the condition $a_{ij}=\delta_{ij}$ by deriving sharp function estimates for the solution operator and its derivatives, combined with Fefferman-Stein and Hardy-Littlewood theorems. However, the sharp function estimates in \cite{Han} rely on the scaling properties of the fundamental solution operator, which fail for nonlocal space operators $\phi(\Delta)$. For the time-space nonlocal equation:
\begin{align}\label{time-space nonlocal equation}
\partial^{\alpha}_{t}w(t,x)=\phi(\Delta)w(t,x)+h(t,x),
\end{align}
Kim \cite{Kim1} established BMO estimates for the solution operator by exploiting its properties. Applying Calder\'{o}n-Zygmund theory, mixed-norm $L_{q}(L_{p})$ priori estimates were obtained for the derivatives of the solution operator \cite[Theorem 4.10]{Kim1}, leading to the unweighted $L_{q}(L_{p})$ theory for \eqref{time-space nonlocal equation}. Notably, in \cite{Dong 2}, using the mean oscillation method from \cite{Dong 1} but replacing finite cylinders $Q_{r}(t_{0},x_{0})$ with infinite cylinders $\left(-\infty,t_{0}\right)\times B_{r}(x_{0})$, the authors derived weighted $L_{q}(L_{p})$ theory for \eqref{volterra equation} under the same assumptions as \cite{Dong 1}. \cite{He3} also made further supplement to \cite{Dong 2}. However, when considering nonlocal operators $\phi(\Delta)$, the method in \cite{Dong 2} only applies to local time derivatives and fails for Caputo time-fractional derivatives $\partial^{\alpha}_{t}$; see \cite[Remark 4.2]{Dong}.

To our knowledge, research on weighted $L_{q}(L_{p})$ theory for time-space nonlocal operators remains limited. Building upon previous developments in the field \cite{Han,Kim1,Kim3,Choi J,Dong,Dong 2}, we establish a comprehensive $L_{q}(L_{p})$ theory with Muckenhoupt weights for time-space fractional evolution equations (TSFEs). Our work provides substantial generalizations in several key aspects:

First, We provide an alternative proof (Lemma \ref{independent estimate}) for Kim's result \cite[Theorem 4.10]{Kim1}. Unlike previous methods that rely on maximal function estimates, we demonstrate that the symbol of $\phi(\Delta)\mathscr{L}_{0}$ is a Marcinkiewicz multiplier in $L_{p}(\mathbb{R}^{d+1})$. By using the Marcinkiewicz multiplier theorem, we directly establish the strong type $L_{p}(\mathbb{R};L_{p}(\mathbb{R}^{d}))$ to $L_{p}(\mathbb{R};L_{p}(\mathbb{R}^{d}))$ boundedness. Subsequently, employing Banach-valued Calder\'{o}n-Zygmund decomposition, we prove weak-type $(1,1)$ estimates, and through Marcinkiewicz interpolation and duality arguments, we recover the main result. This approach offers a more streamlined proof compared to Kim \cite[Theorem 4.10]{Kim1}.

Second, leveraging fundamental solution properties under low-scaling conditions on Bernstein functions, we derive sharp function estimates for derivatives of the fundamental solution to TSFEs. By combining the Fefferman--Stein theorem with Hardy--Littlewood maximal function theory, we establish the weighted $L_{q}(L_{p})$ theory for TSFEs. Notably, our approach circumvents the reliance on scaling properties of solution operators \cite{Han}, making it applicable to a broader class of nonlocal operators than previous methods.

Finally, to handle the presence of Muckenhoupt weights, we develop new techniques for initial value characterization. Since weighted settings preclude direct estimates for certain frequency-localized functions \cite[Lemma 5.2]{Kim1}, we employ trace theory for initial values in weighted Besov spaces $B^{\phi,(W\circ k^{\star})^{1/q}(\gamma+2,\gamma)}_{p,q}(\mu_{1})$. This framework naturally reduces to known results in unweighted cases while providing proper characterization in weighted settings. Our work thus complements and extends existing theories, offering a unified treatment of both weighted and unweighted scenarios for time-space nonlocal evolution equations.

The subsequent sections are arranged as: Section 2 develops the notation, definitions, and key lemmas used throughout the paper, including fractional derivatives and integrals, sharp maximal functions, Hardy-Littlewood maximal functions, and the H\"{o}rmander multiplier theorem. Section 3 establishes pointwise estimates for the sharp maximal function of the nonlocal $\phi(\Delta)$ derivative of solution operators. The main results are presented in Section 4.

\section{Preliminaries}
In this section, we introduce some notations, definitions, and lemmas that will be used throughout this paper.

In this paper, we denote by $C(a,b)$ a generic positive constant depending on constants $a$ and $b$, which may vary from line to line in the subsequent text. We use $*$ and $\star$ to denote the convolution to $t$ and $x$, respectively. Let $(\mathbb{R}^d,|\cdot|)$ denote $d$-dimensional Euclidean space. For $M \subset \mathbb{R}^d$, $|M|$ represents its Lebesgue measure. We define the ball $B_\delta(x) := \{z \in \mathbb{R}^d : |x-z| < \delta\}$ with $B_\delta := B_\delta(0)$. For a multi-index $\gamma = (\gamma_1,...,\gamma_d)$, we set
\[
\frac{\partial}{\partial_{x_{i}}}w=\nabla_{x_{i}}w, \quad \text{and} \quad \nabla^{\gamma}_{x}w=\nabla^{\gamma_{1}}_{x_{1}}\nabla^{\gamma_{2}}_{x_{2}}\dots\nabla^{\gamma_{d}}_{x_{d}}w.
\]
The space $L_{p}(M,\nu,X)$ consists of all $\nu$-measurable functions on $M$ taking values in a Banach space $X$, such that
\[
\int_{M} \|w\|_{X}^{p} \, d\nu < \infty,
\]
and we write $L_{p}(M,\nu,\mathbb{R})=L_{p}(M,\nu)$ for simplicity. The integral average of a measurable function $w$ over $M$ is denoted by $[w]_{M}$, that is,
\[
(w)_{M}=\fint_{M}\|w\|_{X}\,d\nu=\frac{1}{\nu(M)}\int_{M}\|w\|_{X}\,d\nu.
\]
We denote by $\mathcal{S}$ the Schwartz space of rapidly decreasing smooth functions, and by $\mathcal{S}'$ the space of tempered distributions, which is the dual space of $\mathcal{S}$. The Fourier transform and the inverse Fourier transform are denoted by $\mathcal{F}$ and $\mathcal{F}^{-1}$, respectively. For any $g\in\mathcal{S}$,
\[
\mathcal{F}(g)(\xi)=\int_{\mathbb{R}^{d}}g(x)e^{ix\cdot\xi}\,dx, \quad \text{and} \quad \mathcal{F}^{-1}(g)(x)=\int_{\mathbb{R}^{d}}g(\xi)e^{ix\cdot\xi}\,d\xi.
\]
By duality, we can extend the Fourier transform to tempered distributions. Specifically, for $g\in\mathcal{S}'$,
\[
\langle\mathcal{F}g,w\rangle=\langle g,\mathcal{F}w\rangle, \quad \text{and} \quad \langle\mathcal{F}^{-1}g,w\rangle=\langle g,\mathcal{F}^{-1}w\rangle \quad \text{for all } w\in\mathcal{S}.
\]
For a measurable function $h$ with polynomial growth at infinity, we define the pseudo-differential operator $h(D)$ as follows: for $g\in \mathcal{S}$,
\[
h(D)g(x):=\mathcal{F}^{-1}\big[h\cdot\mathcal{F}g(\xi)\big], \quad \text{where } D=i\nabla.
\]
Next, we introduction the definition of fractional integral and Caputo derivative.
\begin{definition}
For the function $w\in L^{1}(0,\infty;\mathcal{S})$, the $0<\alpha<1$ fraction integral $J_{t}^{\alpha}$ and the Caputo fractional derivative $\partial_{t}^{\alpha}$~is pointwise defined as
\begin{align*}
J_{t}^{\alpha}w(t,x)&=g_{\alpha}*w(t,x),\\
\partial^{\alpha}_{t}w(t,x)&=D^{\alpha}_{t}\left(w(t,x)-w(0,x)\right)=\frac{d}{dt}\left(g_{1-\alpha}*(w(t,x)-w(0,x))\right),
\end{align*}
where $D^{\alpha}_{t}$ is Riemann-Liouville derivative.
Moreover, if $w(t,x)$ is absolutely continuous~w.r.t. $t$, we also have
$$
\partial^{\alpha}_{t}w(t,x)=g_{1-\alpha}*\frac{d}{dt}w(t,x),\text{ where }g_{\alpha}(t)=t^{\alpha-1}/\Gamma(\alpha).
$$
\end{definition}
The Mittag-Leffler function ~$E_{\alpha,\beta}(z)$~ plays an important role in the field of fractional differential equations and is defined as
$$
E_{\alpha,\beta}(z)= \sum_{k=0}^{\infty} \frac{z^{n}}{\Gamma(\alpha k + \beta)}, \qquad \text{for }\alpha, \beta > 0, \text{ and }z \in \mathbb{C},
$$
and the following~properities~can refer to \cite{Y zhou2,A.A.Kilbas,Podlubny}.
\begin{proposition}\label{M.T.F.P}For $0<\alpha<1$ and $\rho>0$, we have that
\begin{enumerate}[\rm(i)]
  \item
  \begin{equation}\label{Laplace of Mittag}
    \int_{0}^{\infty} e^{-zt} t^{\beta-1} E_{\alpha,\beta}(\pm \rho t^{\alpha}) \, dt = \frac{z^{\alpha-\beta}}{s^{\alpha} \mp \rho} \quad \text{for} \quad \Re(z) > 0, \, \rho \in \mathbb{C}, \, |z^{-\alpha} \rho| < 1.
    \end{equation}
  \item
  $$
  \frac{d}{dt}\big(t^{\alpha-1}E_{\alpha,\alpha}(-\rho t^{\alpha})\big)=t^{\alpha-2}E_{\alpha,\alpha-1}(-\rho t^{\alpha}),\text{ }t>0.
  $$
  \item
  $$
  \frac{d}{dt}E_{\alpha,1}(-\rho t^{\alpha})=-\rho t^{\alpha-1}E_{\alpha,\alpha}(-\rho t^{\alpha}),\text{ }t>0.
  $$
\end{enumerate}
\end{proposition}
Let $\phi: (0,\infty) \to (0,\infty)$ be a function such that $\lim_{x \to 0^{+}} \phi(x) = 0$ and it fulfills\[
\phi(x)=ax+\int_{(0,\infty)}\big(1-e^{-tx}\big)w(dt),
\]
where $a\geq 0$, and $w$ is a L\'{e}vy measure satisfying $\int_{(0,\infty)}\min\{1,t\}w(dt)<\infty$. Such functions are called Bernstein functions. For $\phi$, we easily observe that
\begin{align}\label{B.S.T,func.bdd}
|x^n \phi^{(n)}(x)| \leq a I_{n=1} + \int_{0}^{\infty} (t x)^n e^{-t x} w(dt) \lesssim \phi(x).
\end{align}

Every Bernstein function $\phi$ corresponds to a subordinator $S_t$ with Laplace transform $\mathbb{E}[e^{-xS_t}] = e^{-t\phi(x)}$. When coupled with an independent $d$-dimensional Brownian motion $W_t$, the composition $Y_t = W_{S_t}$ yields a subordinate Brownian motion in $\mathbb{R}^d$ with characteristic function $\exp(-t\phi(|\xi|^{2}))$ and possesses a transition density $p_{d}(t,x)$ given by
\begin{align*}
\mathbb{E}\exp(ixY_{t})&=\int_{\Omega}\exp(-|\xi|^{2}S_{t})\mathbb{P}(\,d\omega)=\exp(-t\phi(|\xi|^{2})),\\
p_{d}(t,x)&=\mathcal{F}^{-1}\big(\exp(-t\phi(|\xi|^{2}))\big)=\int_{(0,\infty)}\frac{1}{(4\pi s)^{\frac{d}{2}}}\exp\left(-\frac{|x|^{2}}{4s}\right)\zeta_{t}(\,ds),
\end{align*}
where $\zeta_{t}$ denotes the distribution function of $S_{t}$.

The generator of $Y_t$ is given by $\phi(\Delta)$, that is, for $g\in\mathcal{S}$,
\[
\phi(\Delta)g(x)=\lim_{t\rightarrow 0^{+}}\frac{\mathbb{E}g(x+Y_{t})-g(x)}{t}.
\]
This is equivalent to the following integro-differential representation:
\[
\phi(\Delta)g(x)=a\Delta g(x)+\int_{\mathbb{R}^{d}}\big(g(x+z)-g(x)-\nabla g(x)\cdot zI_{|z|\leq 1}\big)j(|z|)\,dz,
\]
where the jump kernel $j$ is given by
\[
j(|z|)=\int_{(0,\infty)}(4\pi t)^{-\frac{d}{2}}\exp\left(-\frac{|z|^{2}}{4t}\right)w(\,dt).
\]
Alternatively, the operator can be defined via the Fourier transform:
\[
\phi(\Delta)g(x) = \mathcal{F}^{-1}\left(-\phi(|\xi|^{2})\mathcal{F}g(\xi)\right)(x).
\]

From the concavity of Bernstein functions, it follows that for any $0<m<M$,
$
\frac{\phi(M)}{M}\leq\frac{\phi(m)}{m}.
$
Moreover,  this work adopts the subsequent lower scaling constraint for $\phi$ as introduced in Kim \cite{Kim2,Kim3}:

\textbf{Assumption}.
$\exists$ $\delta_{0}\in(0,1]$ and $c_{1}>0$ s.t.
\[
c_{1}\left(\frac{M}{m}\right)^{\delta_{0}}\leq\frac{\phi(M)}{\phi(m)},\quad \text{for all } 0<m<M<\infty.
\]
Consequently, we obtain the two-sided inequality
\begin{align}\label{lower scailing condition}
c_{1}\left(\frac{M}{m}\right)^{\delta_{0}}\leq\frac{\phi(M)}{\phi(m)}\leq\frac{M}{m},\quad \text{for all } 0<m<M<\infty.
\end{align}

Several functions satisfy this assumption, including $\phi(x)=x^{\beta}$ and $\phi(x)=x/(\log(1+x^{\frac{\beta}{2}}))$ for $0<\beta<2$, among others (see Kim \cite{Kim1,Kim2}).

\begin{remark}[\cite{Kim2}]
The lower scaling condition implies that $\phi$ cannot grow too slowly at infinity. This condition is crucial for obtaining sharp estimates of the heat kernel associated with $\phi(\Delta)$.
\end{remark}
From above assumption,~it is easy to obtain that
\begin{align}\label{Pro.Convergence}
\int_{\varrho^{-1}}^{\infty}t^{-1}\phi(t^{-2})\,dt=\int_{1}^{\infty}t^{-1}
\frac{\phi(\varrho^{2}t^{-2})}{\phi(\varrho^{2})}\phi(\varrho^{2})\,dt\leq C\int_{1}^{\infty}t^{-1-2\delta_{0}}\,dt
\phi(\varrho^{2})\leq C\phi(\varrho^{2}).
\end{align}
Next we introduce some~facts about Muckenhoupt weights, see \cite{Grafakos}.
\begin{definition}[$A_{p}$ weight]
For $1<p<\infty$ and the nonnegative measurable function $\mu(x)$, we denote that $\mu\in A_{p}(\mathbb{R}^{d})$ if
$$
[\mu]_{p}=\sup_{x_{0}\in\mathbb{R}^{d},\varrho>0}\bigg(\fint_{B_{\varrho}(x_{0})}\mu(x)\,dx\bigg)
\bigg(\fint_{B_{\varrho}(x_{0})}\mu(x)^{-\frac{1}{p-1}}\,dx\bigg)^{p-1}<\infty.
$$
If $\mu\in A_{p}(\mathbb{R}^{d})$, we also say $\mu$ is a $A_{p}$ weight.
\end{definition}
Let $T\in(0,\infty]$,~$1<p,q<\infty$,  $\mu_{1}\in A_{p}(\mathbb{R}^{d})$, $\mu_{2}\in A_{q}(\mathbb{R})$, we~denote that $L_{p}(\mu_{1})=L_{p}\big(\mathbb{R}^{d},\mu_{1}dx\big)$ and for $s\in\mathbb{R}$, denote $H^{s,\phi}_{p}(\mu_{1})$~that the set of termpered distribution satisfying that
$$
\big\|w\big\|_{H^{s,\phi}_{p}(\mu_{1})}=\big\|\big(I-\phi(\Delta)\big)^{\frac{s}{2}}w\big\|_{L_{p}(\mu_{1})}<\infty.
$$
If $\mu_{1}\equiv 1$, then $L_{p}(\mu_{1})=L_{p}(\mathbb{R}^{d})$ and $H^{s,\phi}_{p}(\mu_{1})=H^{s,\phi}_{p}(\mathbb{R}^{d})$ which have been studied in Farkas \cite{Farkas}. Thus, we can define $\mathcal{L}_{p,q}(\mu_{1},\mu_{2},T)$ and $\mathcal{H}_{p,q}^{s,\phi}(\mu_{1},\mu_{2},T)$ that, respectively
\begin{align*}
\mathcal{L}_{p,q}(\mu_{1},\mu_{2},T)=L_{q}\big((0,T),\mu_{2}dt;L_{p}(\mu_{1})\big),\text{ }\mathcal{H}_{p,q}^{s,\phi}(\mu_{1},\mu_{2},T)=L_{q}\big((0,T),\mu_{2}dt;H^{s,\phi}_{p}(\mu_{1})\big),
\end{align*}
and if $T=\infty$, we denote that $\mathcal{L}_{p,q}(\mu_{1},\mu_{2})$ and $\mathcal{H}_{p,q}^{s,\phi}(\mu_{1},\mu_{2})$, respectively.
\begin{remark}\label{Ap weight proposition}
\begin{enumerate}[\rm(i)]For the $A_{p}$ weight function $\mu(x)$, it holds the following proposition:
 \item
 $$
 A_{p}(\mathbb{R}^{d})=\bigcup_{1<q<p}A_{q}(\mathbb{R}^{d}).
 $$
 \item $\mathcal{S}(\mathbb{R}^{d})$ is dense in $L_{p}(\mu)$ and denote that $\tilde{\mu}(x)=\mu^{-\frac{1}{p-1}}$, then $\big(L_{p}(\mu)\big)^{*}=L_{p}(\tilde{\mu})$, that is for any bounded function $T$ on $L_{p}(\mu)$, there exists a unique $h\in L_{p}(\tilde{\mu})$ s.t.
     $$
     Tf=\int_{\mathbb{R}^{d}}f(x)g(x)dx,\text{ for any }f\in L_{p}(\mu).
     $$
 \item For $\psi\in\mathcal{S}$, support in the unit ball $B_{1}(0)$, and denote $\psi_{\varepsilon}(x)=\varepsilon^{-d}\psi(x/\varepsilon)$, then we have that for any $g\in L_{p}(\mu)$,
  $$
  \left\|g\star \psi_{\varepsilon}\right\|_{L_{p}(\mu)}\leq C_{0}\left\|g\right\|_{L_{p}(\mu)},
  $$
  where the $C_{0}$ depend that $d$,~$[\mu]_{p}$,~$p$,~and~$\|\mu\|_{L_{1}}+\|\mu\|_{L_{\infty}}+\|D\mu\|_{L_{\infty}}$.~Moreover,
  $$
  \left\|g\star \psi_{\varepsilon}-g\right\|_{L_{p}(\mu)}\rightarrow 0\text{ as }\varepsilon\rightarrow 0^{+},
  $$
  see that~\rm\cite{Grafakos,Han}.
\end{enumerate}
\end{remark}
Let $\hat{\psi}\in\mathcal{S}$, and support in a strip $\big\{\xi:\frac{1}{2}<|\xi|<2\big\}$, and for $f\in\mathcal{S}'$, denote that
\begin{align*}
\hat{\psi}_{j}(\xi)&=\hat{\psi}(2^{-j}\xi),\text{ }\hat{\Phi}(\xi)=1-\sum_{j=0}^{\infty}\hat{\psi}_{j}(\xi),\\
\Delta_{j}f(x)&=\mathcal{F}^{-1}\big(\hat{\psi}(2^{-j}\xi)\hat{f}(\xi)\big)(x),\text{ }
Sf(x)=\mathcal{F}^{-1}\big(\hat{\Phi}(\xi)\hat{f}(\xi)\big)(x),
\end{align*}
by the Mikulevi\v{c}ius \cite{Mikulevicius}, the space $H^{s,\phi}_{p}(\mathbb{R}^{d})$ admits the following Littlewood-Paley characterization:
$$
\big\|((I-\phi(\Delta))^{\frac{s}{2}}w\big\|_{L_{p}}\sim \left(\|Sw\|_{L_{p}}+\left\|\big\{\phi(2^{2j})^{\frac{s}{2}}|\Delta_{j}w|\big\}\right\|_{L_{p}(l_{2})}\right).
$$
\begin{lemma}\label{weighted Lp Littlewood-Paley}
Let $1<p<\infty$, $s\in\mathbb{R}$, $\mu\in A_{p}(\mathbb{R}^{d})$. For the space $H^{s,\phi}_{p}(\mu)$, the following Littlewood-Paley characterization holds:
\begin{align*}
\left\|w\right\|_{H^{s,\phi}_{p}(\mu)}
\sim\left(\|Sw\|_{L_{p}(\mu)}+
\left\|\bigg(\sum_{j=0}^{\infty}\phi(2^{2j})^{\frac{s}{2}}|\Delta_{j}w|^{2}\bigg)^{\frac{1}{2}}\right\|_{L_{p}(\mu)}\right).
\end{align*}
\end{lemma}
\begin{proof}
Let $\hat{\eta}_{j}=\hat{\psi}_{j-1}+\hat{\psi}_{j}+\hat{\psi}_{j+1}$. Note that for any multi-index $\gamma$,
\begin{align*}
\left|D^{\gamma}_{\xi}\phi(|\xi|^{2})\right|=\left|\sum_{\frac{|\gamma|}{2}\leq j\leq|\gamma|}\phi^{(j)}(|\xi|^{2})\prod_{i=1}^{d}|\xi|^{\gamma_{i}}\right|
\leq C(\gamma,d)\left|\xi\right|^{-|\gamma|},\text{ where }\sum_{i=1}^{d}\gamma_{i}=2j-|\gamma|.
\end{align*}
Consequently, we obtain
\begin{align*}
\left|D^{\gamma}_{\xi}\bigg(\frac{\hat{\eta}_{j}(\xi)\hat{\psi}_{j}(\xi)
(1+\phi(|\xi|^{2}))^{\frac{\sigma}{2}}}{\phi(2^{2j})^{\frac{\sigma}{2}}}\bigg)\right|\leq C(\gamma,d)\left|\xi\right|^{-|\gamma|}.
\end{align*}
The remainder of the proof follows similarly to the argument in \cite[Theorem 3.2]{Park}.
\end{proof}
\begin{definition}
Consider parameters $0<\alpha<1$, $T<\infty$, $1<p,q<\infty$, with Muckenhoupt weights $\mu_1\in A_p(\mathbb{R}^d)$, $\mu_2\in A_q(\mathbb{R})$, and $\sigma\in\mathbb{R}$. If a smooth approximating sequence $\{w_n\}\subset C^\infty([0,\infty)\times\mathbb{R}^d)$ exists such that
\begin{align*}
\left\|w_{n}-w_{m}\right\|_{\mathcal{H}^{\phi,\sigma+2}_{p,q}(\mu_{1},\mu_{2},T)}\rightarrow 0,\text{ and }\left\|\partial^{\alpha}_{t}w_{n}-\partial^{\alpha}_{t}w_{m}\right\|_{\mathcal{H}^{\phi,\sigma}_{p,q}(\mu_{1},\mu_{2},T)}\rightarrow 0\text{ as }m,n\rightarrow\infty,
\end{align*}
we denote the $\partial^{\alpha}_{t}w$ that
$$
\partial^{\alpha}_{t}w=\lim_{n\rightarrow\infty}\partial^{\alpha}_{t}w_{n},\text{ in the sence of }\mathcal{H}^{\phi,\sigma}_{p,q}(\mu_{1},\mu_{2},T),
$$
and we say the $w\in\mathcal{H}^{\alpha,\phi,\sigma+2}_{p,q}(\mu_{1},\mu_{2},T)$ and with the norm
\begin{align}\label{solution space norm}
\left\|w\right\|_{\mathcal{H}^{\alpha,\phi,\sigma+2}_{p,q}(\mu_{1},\mu_{2},T)}
=\left\|w\right\|_{\mathcal{H}^{\phi,\sigma+2}_{p,q}(\mu_{1},\mu_{2},T)}
+\left\|\partial^{\alpha}_{t}w\right\|_{\mathcal{H}^{\phi,\sigma}_{p,q}(\mu_{1},\mu_{2},T)},
\end{align}
and we say the sequence $\{w_{n}\}$ is the defining sequence of $w$. Moreover, if $\{w_{n}\}$ also satisfy that $w_{n}(0,x)=0$ for any $n\in\mathbb{N}$, then we say $w\in \mathcal{H}^{\alpha,\phi,\sigma+2}_{p,q,0}(\mu_{1},\mu_{2},T)$ and equip the same norm with in $\mathcal{H}^{\alpha,\phi,\sigma+2}_{p,q}(\mu_{1},\mu_{2},T)$.
\end{definition}
\begin{lemma}\label{Pro.solution space}
For the space $\mathcal{H}^{\alpha,\phi,\sigma+2}_{p,q}(\mu_{1},\mu_{2},T)$, the following propositions hold:
\begin{enumerate}[\rm(i)]
  \item The spaces $\mathcal{H}^{\alpha,\phi,\sigma+2}_{p,q}(\mu_{1},\mu_{2},T)$ and $\mathcal{H}^{\alpha,\phi,\sigma+2}_{p,q,0}(\mu_{1},\mu_{2},T)$ are Banach spaces when equipped with the norm \eqref{solution space norm}.

  \item For $\gamma\in\mathbb{R}$, the operator $\big(I-\phi(\Delta)\big)^{\frac{\gamma}{2}}$:
  $$
  \mathcal{H}^{\alpha,\phi,\sigma+2}_{p,q}(\mu_{1},\mu_{2},T)\rightarrow
  \mathcal{H}^{\alpha,\phi,\sigma+2-\gamma}_{p,q}(\mu_{1},\mu_{2},T),\text{ }
  \mathcal{H}^{\alpha,\phi,\sigma+2}_{p,q,0}(\mu_{1},\mu_{2},T)\rightarrow
  \mathcal{H}^{\alpha,\phi,\sigma+2-\gamma}_{p,q,0}(\mu_{1},\mu_{2},T)
  $$
  acts as an isometry in each case, with $\partial^{\alpha}_{t}\big(I-\phi(\Delta)\big)^{\frac{\gamma}{2}}w=\big(I-\phi(\Delta)\big)^{\frac{\gamma}{2}}\partial^{\alpha}_{t}w$.

  \item The space $C^{\infty}_{c}\big((0,\infty)\times\mathbb{R}^{d}\big)$ is dense in $\mathcal{H}^{\alpha,\phi,\sigma+2}_{p,q,0}(\mu_{1},\mu_{2},T)$.
\end{enumerate}
\end{lemma}

\begin{proof}
The proof follows arguments similar to those in \cite{Kim2,Han}. We focus only on verifying (iii).

We begin by proving the density of $C^{\infty}_{c}((0,\infty)\times\mathbb{R}^{d})$ in $\mathcal{H}^{\alpha,\phi,2}_{p,q,0}(\mu_{1},\mu_{2},T)$. For any $w\in\mathcal{H}^{\alpha,\phi,2}_{p,q,0}(\mu_{1},\mu_{2},T)$, we may assume without loss of generality that $w\in C^{\infty}([0,\infty)\times\mathbb{R}^{d})$ with $w(0,x)=0$, extended by zero for $t<0$. Furthermore, by multiplying with a spatial cutoff function, we can ensure $w$ has compact support in $x$.

Let $\zeta\in C^{\infty}_{c}\big((1,2)\big)$ and define the regularizations:
\begin{align*}
w^{\varepsilon}(t,x)&=\zeta_{\varepsilon}*w=\varepsilon^{-1}\int_{\mathbb{R}}w(s,x)\zeta\left(\frac{t-s}{\varepsilon}\right)ds,\\
h^{\varepsilon}(t,x)&=\zeta_{\varepsilon}*h=\varepsilon^{-1}\int_{\mathbb{R}}h(s,x)\zeta\left(\frac{t-s}{\varepsilon}\right)ds,\text{ where }h=\partial^{\alpha}_{t}w.
\end{align*}
These satisfy $w^{\varepsilon}=h^{\varepsilon}=0$ for $t<\varepsilon$, with $w^{\varepsilon},h^{\varepsilon}$ being infinitely differentiable in $(t,x)$. Moreover, as $\varepsilon\rightarrow 0$,
\begin{align*}
|w^{\varepsilon}-w|+|\phi(\Delta)w^{\varepsilon}-\phi(\Delta)w|&\rightarrow 0,\\
|h^{\varepsilon}-h|+|\phi(\Delta)h^{\varepsilon}-\phi(\Delta)h|&\rightarrow 0,
\end{align*}
uniformly on $[0,T]\times\mathbb{R}^{d}$.

Now take $\zeta_{1}\in C_{c}^{\infty}(\mathbb{R})$ satisfying $\zeta_{1}(t)=1$ on $[0,T]$ and $\zeta_{1}(t)=0$ for $t\geq T+1$. Define
$$
u^{\varepsilon}(t,x)=\zeta_{1}(t)w^{\varepsilon}(t,x)\in C_{c}^{\infty}\big((0,\infty)\times\mathbb{R}^{d}\big).
$$
For $t\leq T$, the dominated convergence theorem and Fubini's theorem yield:
\begin{align*}
\partial^{\alpha}_{t}u^{\varepsilon}&=\partial^{\alpha}_{t}w^{\varepsilon}=h^{\varepsilon}\rightarrow \partial_{t}^{\alpha}w=h\quad\text{in } \mathcal{L}_{p,q}(\mu_{1},\mu_{2},T),\\
u^{\varepsilon}&\rightarrow w\quad\text{in } \mathcal{H}^{\alpha,\phi,0}_{p,q}(\mu_{1},\mu_{2},T).
\end{align*}
This establishes the density of $C^{\infty}_{c}\big((0,\infty)\times\mathbb{R}^{d}\big)$ in $\mathcal{H}^{\alpha,\phi,2}_{p,q,0}(\mu_{1},\mu_{2},T)$.

For general $w\in\mathcal{H}^{\alpha,\phi,\sigma+2}_{p,q,0}(\mu_{1},\mu_{2},T)$, define
$$
u:=\big(I-\phi(\Delta)\big)^{\frac{\sigma}{2}}w\in\mathcal{H}^{\alpha,\phi,2}_{p,q,0}(\mu_{1},\mu_{2},T).
$$
Therefore, exists $\{u_{n}\}\subset C^{\infty}_{c}\big((0,\infty)\times\mathbb{R}^{d}\big)$  and  $u_{n}\rightarrow u$ in $\mathcal{H}^{\alpha,\phi,2}_{p,q,0}(\mu_{1},\mu_{2},T)$.

Let $v_{n}:=\big(I-\phi(\Delta)\big)^{-\frac{\sigma}{2}}u_{n}\in C^{\infty}\big((0,\infty)\times\mathbb{R}^{d}\big)$ with compact temporal support. The sequence $v_{n}$ converges to $w$ in $\mathcal{H}^{\alpha,\phi,\sigma+2}_{p,q,0}(\mu_{1},\mu_{2},T)$.

Take $\xi\in C_{c}^{\infty}(\mathbb{R}^{d})$ with $\xi=1$ on $B_{1}$ and $\xi=0$ outside $B_{2}$. Defining $v_{n,i}=\xi(x/i)v_{n}$, we obtain:
$$
v_{n,i}\in C_{c}^{\infty}\big((0,\infty)\times\mathbb{R}^{d}\big)\text{ and }v_{n,i}\rightarrow v_{n}\text{ in }\mathcal{H}^{\alpha,2k}_{p,q}(\mu_{1},\mu_{2},T)\text{ as }i\rightarrow\infty.
$$
For sufficiently large $k\in\mathbb{N}$ with $2k\geq\sigma+2$, Lemma \ref{weighted Lp Littlewood-Paley} gives the embeddings:
$$
\mathcal{H}^{\alpha,2k}_{p,q}(\mu_{1},\mu_{2},T)\subset\mathcal{H}^{\alpha,\phi,2k}_{p,q}(\mu_{1},\mu_{2},T)
\subset\mathcal{H}^{\alpha,\phi,\sigma+2}_{p,q}(\mu_{1},\mu_{2},T).
$$
The result follows by taking $i(n)=n$.
\end{proof}
The following integration by parts formula will be used in Section~4.

If $H,F$ are sufficiently smooth, then for any $0<m<M<\infty$,
\begin{equation}\label{part intefration formula}
\begin{split}
\int_{m\leq|x|\leq M}H(x)F(|x|)\,dx &= F(M)\int_{|x|\leq M}H(x)\,dx - F(m)\int_{|x|\leq m}H(x)\,dx\\
&\quad - \int_{m}^{M}F'(\rho)\left(\int_{|x|\leq\rho}H(x)\,dx\right)\,d\rho\\
&= \int_{m}^{M}F(\rho)\frac{d}{d\rho}\left(\int_{B_{\rho}(0)}H(x)\,dx\right)\,d\rho\\
&= \int_{m}^{M}F(\rho)\int_{\partial B_{\rho}(0)}H(x)\,dS\,d\rho.
\end{split}
\end{equation}
\section{Key Estimates}
In this section, we establish several crucial estimates that will play a fundamental role in Section 4.

Given subordinate Brownian motion $Y_t$ and its associated transition probability $p(t,x)$. Consider $X_{t}$ as an independent subordinator with characteristic exponent $\exp(-t\lambda^{\alpha})$, and define its inverse process:
$$
Q_{t} := \inf\big\{s : X_{s} > t\big\}.
$$
Then the subordinate Brownian motion $Y_{Q_{t}}$ possesses the transition probability
$$
\mathcal{S}_{\alpha,\phi}(s,z) = \int_{0}^{\infty} p(z,t) \varpi(s,t) dt,
$$
where $\varpi(t,r)$ denotes the transition probability of the inverse subordinator $Q_{t}$. Indeed, $\mathcal{S}_{\alpha,\phi}(s,z)$ serves as the fundamental solution to the equation:
$$
\partial^{\alpha}_{t} w(s,z) = \phi(\Delta) w(s,z), \quad w(0,z) = \delta_{0}.
$$

For $\gamma \in \mathbb{R}$, we define
$$
\varpi_{\alpha,\gamma}(t,r) = D^{\gamma-\alpha}_{t} \varpi(t,r)
$$
and introduce the generalized solution kernel:
$$
\mathcal{S}_{\alpha,\gamma,\phi}(s,z) = \int_{0}^{\infty} p(z,t) \varpi_{\alpha,\gamma}(s,t) dt, \quad (s,z) \in \mathbb{R}^{+} \times \mathbb{R}^{d} \setminus \{0\}.
$$

The following proposition, whose proof can be found in Kim \cite{Kim1}, provides essential estimates for these objects.
\begin{proposition}\label{Pro.soltioun function}
For the $\mathcal{S}_{\alpha,\gamma,\phi}(t,x)$, the following~proposition~hold:
\begin{enumerate}[\rm(i)]
  \item
  \begin{align*}
  \mathcal{F}(\mathcal{S}_{\alpha,\gamma,\phi})(t,\xi)&=t^{\alpha-\gamma}E_{\alpha,1+\alpha-\gamma}(-t^{\alpha}\phi(|\xi|^{2})),\text{ }t>0,\xi\in\mathbb{R}^{d},\\
  D^{\gamma-\alpha}_{t}\mathcal{S}_{\alpha,\phi}(t,x)&=\mathcal{S}_{\alpha,\gamma,\phi}(t,x),\text{ }(t,x)
  \in\mathbb{R}^{+}\times\mathbb{R}^{d}\setminus\{0\}.
  \end{align*}
  \item For $k\in\mathbb{N}$, there exists the constant $C(k,\alpha,\gamma,\delta_{0},d)$ s.t.
  \begin{align*}
  \left|\partial^{k}_{x}\mathcal{S}_{\alpha,\gamma,\phi}(t,x)\right|\leq Ct^{2\alpha-\gamma}\frac{\phi(|x|^{-2})}{|x|^{d+k}}.
  \end{align*}
  Moreover, if $t^{\alpha}\phi(|x|^{-2})\geq 1$,~it holds that
  \begin{align*}
  \left|\partial^{k}_{x}\mathcal{S}_{\alpha,\gamma,\phi}(t,x)\right|\leq C\int_{(\phi(|x|^{-2}))^{-1}}^{2t^{\alpha}}\big(\phi^{-1}(\rho^{-1})\big)^{\frac{d+k}{2}}t^{-\gamma}\,d\rho.
  \end{align*}
  \item
  \begin{align*}
  \int_{\mathbb{R}^{d}}\sup_{[\varepsilon,T]}\left|\mathcal{S}_{\alpha,\gamma,\phi}(t,x)\right|\,dx<\infty,\text{ for any }0<\varepsilon<T<\infty,
  \end{align*}
  and
  \begin{align*}
  \int_{\mathbb{R}^{d}}\left|\mathcal{S}_{\alpha,\gamma,\phi}(t,x)\right|\,dx\leq Ct^{\alpha-\gamma}.
  \end{align*}
\end{enumerate}
\end{proposition}
\begin{lemma}[see \cite{Kim1,Kim3}]\label{Solution Resp}
The following facts hold:
\begin{enumerate}[\rm(i)]
\item
For $w\in C_{c}^{\infty}\big((0,\infty)\times\mathbb{R}^{d}\big)$ and let $h=\partial^{\alpha}_{t}w-\phi(\Delta)w$, then it holds
\begin{align}\label{Rep.solution}
w(t,x)=\int_{0}^{t}\int_{\mathbb{R}^{d}}\mathcal{S}_{\alpha,1,\phi}(t-\tau,x-y)h(\tau,y)\,dyd\tau.
\end{align}
\item
Let $h\in C^{\infty}_{c}\big((0,\infty)\times\mathbb{R}^{d}\big)$ and $w(t,x)$ is defined as \eqref{Rep.solution}, then $w(t,x)$ satisfy
\eqref{TSFE} pointwise.
\end{enumerate}
\end{lemma}
\begin{remark}
For any $T>0$,~Note that $\mathcal{S}_{\alpha,1,\phi}(t,x)$ is integrable on $[0,T]\times\mathbb{R}^{d}$, thus the operator $\mathscr{L}_{0}$ and $\mathscr{L}$ are well defined on $C^{\infty}_{c}\big((0,\infty)\times\mathbb{R}^{d}\big)$,
\begin{align*}
\mathscr{L}_{0}h:=\int_{-\infty}^{t}\int_{\mathbb{R}^{d}}\mathcal{S}_{\alpha,1,\phi}(t-\tau,x-y)h(\tau,y)\,dyd\tau,
\end{align*}
\begin{align*}
\mathscr{L}h:=\int_{-\infty}^{t}\int_{\mathbb{R}^{d}}\mathcal{S}_{\alpha,1,\phi}(t-\tau,y)\phi(\Delta)h(\tau,x-y)\,dyd\tau.
\end{align*}
Moreover, from {\rm\cite{Kim1}},
\begin{align*}
\mathscr{L}h=\phi(\Delta)\mathscr{L}_{0}h:=\lim_{\varepsilon\rightarrow 0}\int_{-\infty}^{t-\varepsilon}\int_{\mathbb{R}^{d}}
\mathcal{S}_{\alpha,1+\alpha,\phi}(t-\tau,y)h(\tau,x-y)\,dyd\tau.
\end{align*}
\end{remark}
\begin{lemma}
Let $1<p,q<\infty$, $h\in L_{q}(\mathbb{R};L_{p}(\mathbb{R}^{d}))$ and let $w$ satisfy the time-space fractional equations (TSFEs) \eqref{TSFE}. Then it holds that
\begin{align}\label{independent estimate}
\left\|\phi(\Delta)w\right\|_{L_{q}(\mathbb{R};L_{p}(\mathbb{R}^{d}))}
\leq C(\alpha,d,p,q,\delta_{0})\left\|h\right\|_{L_{q}(\mathbb{R};L_{p}(\mathbb{R}^{d}))}
\end{align}
\end{lemma}
\begin{remark}
In the proof of Kim \cite[Theorem 4.10]{Kim1}, the establishment of \eqref{independent estimate} primarily relies on the sharp maximum function estimate for the operator $\mathscr{L}$, as detailed in \cite[Lemmas 4.3-4.7]{Kim1}. Here, we provide an alternative proof, which mainly relies the Marcinkiewicz multiplier theory on $L_{p}(\mathbb{R}^{d+1})$. The proof of Lemma {\rm\ref{independent estimate}} differs from that of Kim {\rm\cite[Theorem 4.10]{Kim1}}.
\end{remark}
\begin{proof}
The proof relies the Marcinkiewicz multiplier theory on $L_{p}(\mathbb{R}^{d+1})$.
Since $C_{c}^{\infty}(\mathbb{R}\times\mathbb{R}^{d})$ is dense in $L_{q}(\mathbb{R};L_{p}(\mathbb{R}^{d}))$, it suffices to verify \eqref{independent estimate} for $h\in C^{\infty}_{c}(\mathbb{R}\times\mathbb{R}^{d})$.

From Lemma \ref{Rep.solution} and Proposition \ref{Pro.soltioun function}, we have
\begin{align*}
w(t,x)=\int_{\mathbb{R}^{d+1}}\textbf{I}_{0<t-\tau}\mathcal{S}_{\alpha,1,\phi}(t-\tau,y)h(\tau,x-y)\,dy\,d\tau,\text{ }
\mathcal{F}(\mathcal{S}_{\alpha,1,\phi})(t,\xi)&=t^{\alpha-1}E_{\alpha,\alpha}(-t^{\alpha}\phi(|\xi|^{2})).
\end{align*}
Thus,
\begin{align*}
\mathcal{F}_{d+1}\big(\phi(\Delta)w\big)(\vartheta,\xi)&=\bigg(\int_{\mathbb{R}}e^{-it\vartheta}\textbf{I}_{t>0}
\mathcal{F}_{d}\big(\phi(\Delta)
\mathcal{S}_{\alpha,1,\phi}\big)(t,\xi)\,dt\bigg)\hat{h}(\vartheta,\xi)\\
&=\bigg(\int_{0}^{\infty}e^{-it\vartheta}\phi(|\xi|^{2})t^{\alpha-1}E_{\alpha,\alpha}(-t^{\alpha}\phi(|\xi|^{2}))\bigg)\hat{h}(\vartheta,\xi)\\
&=\frac{\phi(|\xi|^{2})}{(i\vartheta)^{\alpha}+\phi(|\xi|^{2})}\hat{h}(\vartheta,\xi):=m(\vartheta,\xi)\hat{h}(\vartheta,\xi),
\end{align*}
where we use the Laplace transform of the Mittag-Leffler function \eqref{Laplace of Mittag} with $s=i\vartheta$, $\rho=\phi(|\xi|^{2})$ and $\beta=\alpha$,
$$
\int_{0}^{\infty}e^{-st}t^{\alpha-\beta}E_{\alpha,\beta}(-t^{\alpha}\rho)\,dt=\frac{t^{\alpha-\beta}}{s^{\alpha}+\rho}.
$$
Next, we prove that $m(\vartheta,\xi)$ is a Marcinkiewicz multiplier on $L_{p}(\mathbb{R}^{d+1})$.

Recall that for any multi-index $\gamma$,
$$
D^{\gamma}_{\xi}\phi(|\xi|^{2})=\sum_{\frac{|\gamma|}{2}\leq j\leq|\gamma|}\phi^{(j)}(|\xi|^{2})\prod_{i=1}^{d}|\xi_{i}|^{\beta_{i}},
\quad \text{ where }\sum_{i=1}^{d}\beta_{i}=2j-|\gamma|,
$$
combined with \eqref{B.S.T,func.bdd}, this implies $\big|D^{\gamma}_{\xi}\phi(|\xi|^{2})\big|\lesssim \phi(|\xi|^{2})|\xi|^{-|\gamma|}$. By the Leibniz rule, we obtain for any decomposition $\gamma_{1}+\gamma_{2}=\gamma$ with $\gamma_{2}\neq 0$
\begin{align*}
\bigg|D_{\xi}^{\gamma}\big[m(\vartheta,\xi)\big]\bigg|&\lesssim\sum_{\substack{{\gamma_{1}+\gamma_{2}=\gamma},\\ \beta_{1}+..+\beta_{l}=\gamma_{2},\\ 1\leq l\leq |\gamma_{2}|}}\bigg|D^{\gamma_{1}}_{\xi}\phi(|\xi|^{2})\frac{1}{\big[(i\vartheta)^{\alpha}+\phi(|\xi|^{2})\big]^{l+1}}
\prod_{i=1}^{l}D^{\beta_{i}}_{\xi}\phi(|\xi|^{2})\bigg|\\
&\lesssim\sum_{\substack{{\gamma_{1}+\gamma_{2}=\gamma},\\ \beta_{1}+..+\beta_{l}=\gamma_{2},\\ 1\leq l\leq |\gamma_{2}|}}
\bigg|\phi(|\xi|^{2})|\xi|^{-|\gamma_{1}|}\frac{(\phi(|\xi|^{2}))^{l}|\xi|^{-|\gamma_{2}|}}
{\big[(i\vartheta)^{\alpha}+\phi(|\xi|^{2})\big]^{l+1}}\bigg|\\
&\lesssim\frac{\phi(|\xi|^{2})}{|\vartheta|^{\alpha}+\phi(|\xi|^{2})}\big|\xi\big|^{-|\gamma|},
\end{align*}
and for $\gamma=0$, the estimate obviously holds. Hence, for any $\gamma$, we have
$$
\bigg|D_{\xi}^{\gamma}\big[m(\vartheta,\xi)\big]\bigg|\lesssim
\frac{\phi(|\xi|^{2})}{|\vartheta|^{\alpha}+\phi(|\xi|^{2})}\big|\xi\big|^{-|\gamma|}.
$$
Next, let $\hat{\beta}=\big(\beta_{0},\beta_{1},...\beta_{d}\big)=\big(\beta_{0},\beta\big)$ be a $d+1$ dimensional multi-index with $\beta_{0}\neq 0$, $\beta_{i}=0\text{ or }1$ for $i=1,2,...d$. Thus we have
\begin{align}\label{multiplier upper estimate}
\big|D^{\hat{\beta}}m(\vartheta,\xi)\big|&\lesssim
\frac{|\vartheta|^{\beta_{0}(\alpha-1)}}{\big[|\vartheta|^{\alpha}+\phi(|\xi|^{2})\big]^{\beta_{0}}}\big|D^{\beta}_{\xi}m(\vartheta,\xi)\big|
\lesssim\frac{\vartheta^{\beta_{0}(\alpha-1)}\phi(|\xi|^{2})}{\big[|\vartheta|^{\alpha}+\phi(|\xi|^{2})\big]^{\beta_{0}+1}}|\xi|^{-|\gamma|}.
\end{align}
Hence, from \eqref{multiplier upper estimate}, for any $0<l\leq d+1$, it follows that
$$
\bigg|\frac{\partial^{l}m}{\partial_{\vartheta}\partial_{1}\partial_{2}...\partial_{\xi_{l-1}}}\bigg|\lesssim
\frac{|\vartheta|^{(\alpha-1)}\phi(|\xi|^{2})}{\big[|\vartheta|^{\alpha}+\phi(|\xi|^{2})\big]^{2}}|\xi|^{-|l-1|}.
$$
Note that for any $\zeta>0$,
\begin{align*}
\bigg|\int_{\zeta}^{2\zeta}\frac{\vartheta^{\alpha-1}\phi(|\xi|^{2})}{\big[\vartheta^{\alpha}+\phi(|\xi|^{2})\big]^{2}}\,d\vartheta\bigg|
&\leq\frac{1}{\alpha}\bigg|\int_{\zeta^{\alpha}}^{(2\zeta)^{\alpha}}
\frac{\phi(|\xi|^{2})}{\big[\mu+\phi(|\xi|^{2})\big]^{2}}\,d\mu\bigg|\\
&\leq \frac{1}{\alpha}\bigg(\frac{\phi(|\xi|^{2})}{\zeta^{\alpha}+\phi(|\xi|^{2})}-
\frac{\phi(|\xi|^{2})}{(2\zeta)^{\alpha}+\phi(|\xi|^{2})}\bigg)\\
&\leq\frac{1}{\alpha}\frac{\phi(|\xi|^{2})[(2\zeta)^{\alpha}+\phi(|\xi|^{2})]-\phi(|\xi|^{2})[(\zeta)^{\alpha}+\phi(|\xi|^{2})]}
{\phi(|\xi|^{2})[(2\zeta)^{\alpha}+\phi(|\xi|^{2})]}\\
&\leq \frac{1}{\alpha}\frac{2^{\alpha}-1}{2^{\alpha}},
\end{align*}
and
$$
\int_{\zeta}^{2\zeta}|\xi_{j}|^{-1}\,d\xi_{j}\lesssim 1,\text{ for }j=1,2,...d.
$$
Thus we derive that for any $0<l\leq d+1$,
\begin{align}\label{panju}
\sup_{\xi_{l},\xi_{l+1},...\xi_{d}}\int_{\rho}
\bigg|\frac{\partial^{l}m}{\partial_{\vartheta}\partial_{1}\partial_{2}...
\partial_{\xi_{l-1}}}\bigg|\,d\vartheta\,d\xi_{1}\,...\,d\xi_{l-1}\lesssim 1,
\end{align}
as $\rho$ ranges over dyadic rectangles $\prod_{1\leq i\leq l}[2^{k_{i}},2^{k_{i}+1}]$ of $\mathbb{R}^{l}$. (If $l=d+1$, the $\sup$ sign is omitted.) Moreover, it is easy to see that \eqref{panju} is valid for every one of the $(d+1)!$ permutations of the variables $\vartheta,\xi_{1},...,\xi_{d}$. By the Marcinkiewicz multiplier theory, see \cite[Theorem 4.6']{Stein}, $m(\vartheta,\xi)$ is a Marcinkiewicz multiplier on $L_{p}(\mathbb{R}^{d+1})$ for $1<p<\infty$, which implies that for any $h\in C_{c}^{\infty}(\mathbb{R}^{d+1})$
\begin{align}\label{Lpestimate}
\big\|\phi(\Delta)w\big\|_{L_{p}(\mathbb{R};L_{p}(\mathbb{R}^{d}))}\lesssim\big\|h\big\|_{L_{p}(\mathbb{R};L_{p}(\mathbb{R}^{d}))}.
\end{align}
Next, we denote $X=L_{p}(\mathbb{R}^{d})$, and define the operator
$$
\mathcal{T}(t,\tau)h(x)=\int_{\mathbb{R}^{d}}\mathcal{S}_{\alpha,1+\alpha,\phi}(t-\tau,x-y)h(\tau,y)\,dy.
$$
Using Young's inequality and combining with (iii) of Proposition \ref{Pro.soltioun function},
$$
\big\|\mathcal{T}(t,\tau)h\big\|_{X}\lesssim |t-\tau|^{-1}\big\|h(t)\big\|_{X},
$$
this implies the operator $\mathcal{T}$ is uniquely extendible to $X$ for $t\neq \tau$.
Let
$$
I_{j}=[a2^{j},(a+1)2^{j})\text{ and }I_{j}^{*}=\big((a-1)2^{j},(a+2)2^{j}\big),\text{ where }a,j\in\mathbb{Z},
$$
and recall $\mathcal{S}_{\alpha,1+\alpha,\phi}(t-\tau,x-y)=0$ if $t\leq \tau$. For $h\in C^{\infty}_{c}(\mathbb{R};X)$, we define the operator
$$
\mathscr{T}h(t,x)=\int_{\mathbb{R}}\mathcal{T}(t,\tau)h(\tau,x)\,d\tau.
$$
On one hand, \eqref{Lpestimate} implies the operator $\mathscr{T}$ is of strong type $L_{p}(X)$ to $L_{p}(X)$. On the other hand, for $t\in (I_{j}^{*})^{c}$, $s,\tau\in I_{j}$, combining with (iii) of Proposition \ref{Pro.soltioun function} and using Young's inequality, we derive that
\begin{align*}
\left\|\mathcal{T}(t,\tau)-\mathcal{T}(t,s)\right\|_{\mathcal{L}(L_{p})}
&=\sup_{\|h\|_{p}\leq 1}\bigg\|\int_{\mathbb{R}^{d}}
\big\{\mathcal{S}_{1,1+\alpha,\phi}(t-\tau,x-y)-\mathcal{S}_{1,1+\alpha,\phi}(t-s,x-y)\big\}h(y)\,dy\bigg\|_{L_{p}}\\
&\leq\int_{\mathbb{R}^{d}}\big|\tau-s\big|\big|\mathcal{S}_{1,2+\alpha,\phi}(t-\theta(\tau,s),x)\big|\,dx\\
&\lesssim\frac{|s-\tau|}{(t-((a+1)2^{j}))^{2}},
\end{align*}
where $\theta(\tau,s)=\mu\tau+(1-\mu)s\in I_{j}$, $\mu\in(0,1)$, and
\begin{align}\label{verify condition}
&\int_{|t-\tau|>2|s-\tau|}\left\|\mathcal{T}(t,\tau)-\mathcal{T}(t,s)\right\|_{\mathcal{L}(L_{p})}\,dt\notag\\
&\lesssim\int_{\mathbb{R}\backslash I_{j}^{*}}\left\|\mathcal{T}(t,\tau)-\mathcal{T}(t,s)\right\|_{\mathcal{L}(L_{p})}\,dt\notag\\
&\lesssim|I_{j}|\int_{|t|\geq2^{j}}|t|^{-2}\,dt\lesssim 1.
\end{align}
By the Banach-valued Calder\'{o}n-Zygmund decomposition \cite[Proposition 11.2.6]{Weis}, for a fixed $\lambda>0$, we can decompose $h=g+b$, where
$$
\big\|g\big\|_{L_{\infty}(\mathbb{R};X)}\leq 2\lambda,\quad\text{ }\big\|g\big\|_{L_{1}(\mathbb{R};X)}\leq\big\|h\big\|_{L_{1}(\mathbb{R};X)},\quad b=\sum b_{j},
$$
$$
\text{supp}\text{ }b_{j}\subseteq I_{j},\quad\int_{I_{j}}b_{j}=0,\quad\sum_{j}|I_{j}|\leq\frac{1}{\lambda}\|h\|_{L_{1}(\mathbb{R};X)},
\quad\sum_{i}\big\|b_{j}\big\|_{L_{1}(\mathbb{R};X)}\leq 2\big\|h\big\|_{L_{1}(\mathbb{R};X)}
$$
for some disjoint dyadic intervals $I_{j}$. We derive that
\begin{align*}
\left|\big\{t:\|\mathscr{T}h(t)\|_{X}>\lambda\big\}\right|\leq\left|\big\{t:\|\mathscr{T}g(t)\|_{X}>\lambda/2\big\}\right|
+\left|\big\{t:\|\mathscr{T}b(t)\|_{X}>\lambda/2\big\}\right|.
\end{align*}
By Chebyshev's inequality and \eqref{Lpestimate}, we obtain that
\begin{align*}
\left|\big\{t:\|\mathscr{T}g(t)\|_{X}>\lambda/2\big\}\right|\leq\frac{2^{p}}{\lambda^{p}}
\int_{\mathbb{R}}\big\|\mathscr{T}g(t)\big\|_{X}^{p}\,dt
\leq\frac{2^{p}}{\lambda^{p}}\int_{\mathbb{R}}\big\|g(t)\big\|_{X}^{p}\,dt
\leq\frac{2^{2p-1}}{\lambda}\left\|h\right\|_{L_{1}(\mathbb{R};X)},
\end{align*}
\begin{align*}
\left|\big\{t:\|\mathscr{T}b(t)\|_{X}>\lambda/2\big\}\right|&\leq\big|\big\{\bigcup_{j}I^{*}_{j}\big\}\big|+
\big|\big\{t\notin\bigcup_{j}I^{*}_{j}:\|\mathscr{T}b(t)\|_{X}>\lambda/2\big\}\big|\\
&\lesssim\frac{2}{\lambda}\|h\|_{L_{1}(\mathbb{R};X)}+\big|\big\{t\notin\bigcup_{j}I^{*}_{j}:\|\mathscr{T}b(t)\|_{X}>\lambda/2\big\}\big|,
\end{align*}
and
\begin{align*}
\big|\big\{t\notin\bigcup_{j}I^{*}_{j}:\|\mathscr{T}b(t)\|_{X}>\lambda/2\big\}\big|&\leq \frac{2}{\lambda}\sum_{j}\int_{\mathbb{R}\backslash I_{j}^{*}}\|\mathscr{T}b_{j}(t)\|_{X}\,dt\\
&\leq\frac{2}{\lambda}\sum_{j}\int_{\mathbb{R}\backslash I_{j}^{*}}\bigg\|\int_{I_{j}}\mathcal{T}(t,\tau)b_{j}(\tau)\,d\tau\bigg\|_{X}\,dt\\
&=\frac{2}{\lambda}\sum_{j}\int_{\mathbb{R}\backslash I_{j}^{*}}\bigg\|\int_{I_{j}}[\mathcal{T}(t,\tau)-\mathcal{T}(t,c_{j})]b_{j}(\tau)\,d\tau\bigg\|_{X}\,dt\\
&\leq\frac{2}{\lambda}\sum_{j}\int_{\mathbb{R}\backslash I_{j}^{*}}\int_{I_{j}}\big\|[\mathcal{T}(t,\tau)-\mathcal{T}(t,c_{j})]\big\|_{\mathcal{L}(X)}\|b_{j}(\tau)\|_{X}\,d\tau\,dt\\
&\lesssim\frac{2}{\lambda}\sum_{j}\int_{I_{j}}\big\|b_{j}(\tau)\big\|_{X}\,d\tau
\lesssim\frac{4}{\lambda}\big\|f\big\|_{L_{1}(\mathbb{R};X)},
\end{align*}
where we use \eqref{verify condition} and the cancellation condition $\int_{I_{j}}b_{j}=0$,
thus we obtain
$$
\left|\big\{t:\|\mathscr{T}h(t)\|_{X}>\lambda\big\}\right|\lesssim\frac{1}{\lambda}\left\|h\right\|_{L_{1}(\mathbb{R};X)},
$$
this implies that the operator $\mathscr{T}$ maps $L_{1}(\mathbb{R};X)$ to $L_{1,\infty}(\mathbb{R};X)$. By the Marcinkiewicz interpolation theorem and \eqref{Lpestimate}, we obtain that the operator maps $L_{q}(\mathbb{R};X)$ to $L_{q}(\mathbb{R};X)$ for $1<q\leq p$.
Moreover, for the case $1<p<q<\infty$, using \cite[(4.12)]{Kim1} and H\"{o}lder's inequality, we get for any $g\in C^{\infty}_{c}(\mathbb{R};X)$, $1<q'<p'$, it follows
\begin{align*}
\bigg|\int_{\mathbb{R}}\big\langle g(t,\cdot),\mathscr{T}h(t,\cdot)\big\rangle_{(X',X)}\,dt\bigg|&=\bigg|\int_{\mathbb{R}}\big\langle \mathscr{T}\tilde{g}(t,\cdot),\tilde{h}(t,\cdot)\big\rangle_{(X',X)}\,dt\bigg|\\
&\leq\int_{\mathbb{R}}\big\|\mathscr{T}\tilde{g}(t)\big\|_{X'}\big\|\tilde{h}(t)\big\|_{X}\,dt\\
&\leq\big\|\mathscr{T}\tilde{g}\big\|_{L_{q'}(\mathbb{R};X')}\big\|\tilde{h}\big\|_{L_{q}(\mathbb{R};X)}\\
&\leq\big\|g\big\|_{L_{q'}(\mathbb{R};X')}\big\|h\big\|_{L_{q}(\mathbb{R};X)},
\end{align*}
hence, in summary, we obtain \eqref{independent estimate} for any $1<p,q<\infty$.
\end{proof}
Now, for any $(t,x)\in\mathbb{R}^{d+1}$ and constant $\varrho>0$, we denote
\begin{align*}
\lambda(\varrho)=\left(\phi(\varrho^{-2})\right)^{-\frac{1}{\alpha}},\text{ }B_{\varrho}(x)=\left\{z:|x-z|<\varrho\right\},
\end{align*}
and
\begin{align*}
I_{\varrho}(t)=\left(t-\lambda(\varrho),t\right),\text{ }\mathcal{Q}_{\varrho}(t,x)=I_{\varrho}(t)\times B_{\varrho}(x),\text{ }\mathcal{Q}_{\varrho}:=\mathcal{Q}_{\rho}(0,0).
\end{align*}
The sharp maximal function associated with a real-valued measurable function $h(t,x)$ on $\mathbb{R}^{d+1}$, denoted by $h^{\sharp}(t,x)$, is defined by
\[
h^{\sharp}(t,x) := \sup_{(t,x)\in\mathcal{Q}} \fint_{\mathcal{Q}} \left|h(s,y) - (h)_{\mathcal{Q}}\right| \,dy\,ds,
\]
where the supremum is taken over all parabolic cylinders $\mathcal{Q}_\varrho \owns (t,x)$. The spatial maximal operator $\mathcal{M}_x$ acts on $g:\mathbb{R}^d\to\mathbb{R}$ as:
\begin{align*}
\mathcal{M}_{x}g(x)=\sup_{x\in B_{\varrho}(z)}\fint_{B_{\varrho}(z)}\left|g(y)\right|\,dy.
\end{align*}
Moreover, the \emph{maximal function} $\mathcal{M}_{t,x}h(t,x)$ can be defined similarly, and
$$
\mathcal{M}_{t}\mathcal{M}_{x}h(t,x)=\mathcal{M}_{t}\left(\mathcal{M}_{x}h(\cdot,x)\right)(t).
$$
For $h\in C_{c}^{\infty}\left(\mathbb{R}^{d+1}\right)$, $(t,x)\in\mathbb{R}^{d+1}$, $0<T<\infty$, we define the operators $\mathcal{G}_{0}$ and $\mathcal{G}_{1}$:
\begin{align*}
\mathcal{G}_{0}h(t,x)=\int_{-\infty}^{t}\int_{\mathbb{R}^{d}}\textbf{I}_{0<t-\tau<T}\mathcal{S}_{\alpha,1,\phi}(t-\tau,x-y)h(\tau,y)\,dyd\tau,
\end{align*}
and
\begin{align*}
\mathcal{G}_{1}h(t,x)=\int_{-\infty}^{t}\int_{\mathbb{R}^{d}}\mathcal{S}_{\alpha,1+\alpha,\phi}(t-\tau,x-y)h(\tau,y)\,dyd\tau.
\end{align*}
\begin{lemma}\label{supp estimate 1}
Let $T<\infty$, $1<p<\infty$, $\varrho>0$, and $h\in C_{c}^{\infty}\left(\mathbb{R}^{d+1}\right)$ with support contained in $\left(-3\lambda(\varrho),3\lambda(\varrho)\right)\times B_{3\varrho}$. Then for any $(t,x)\in\mathcal{Q}_{\varrho}$, the following estimates hold:
\begin{align}
\label{supp estimate 1.1}
\fint_{\mathcal{Q}_{\varrho}}\left|\mathcal{G}_{0}h(s,z)\right|^{p}\,dzds &\leq C(\alpha,d,p,\delta_{0},T)\mathcal{M}_{t,x}|h|^{p}(t,x),\\
\fint_{\mathcal{Q}_{\varrho}}\left|\mathcal{G}_{1}h(s,z)\right|^{p}\,dzds &\leq C(\alpha,d,p,\delta_{0})\mathcal{M}_{t,x}|h|^{p}(t,x).
\label{supp estimate 1.2}
\end{align}
\end{lemma}

\begin{proof}
By the Minkowski inequality, we have
\begin{align*}
&\fint_{\mathcal{Q}_{\varrho}}\left|\mathcal{G}_{0}h(s,z)\right|^{p}\,dzds\\
&\leq C(d)\varrho^{-d}\lambda(\varrho)^{-1}
\left(\int_{0}^{T}\int_{\mathbb{R}^{d}}\left(\int_{\mathcal{Q}_{\varrho}}
\left|\mathcal{S}_{\alpha,1,\phi}(\tau,y)h(s-\tau,z-y)\right|^{p}\,dzds\right)^{\frac{1}{p}}\,dyd\tau\right)^{p}\\
&\leq C(d)\varrho^{-d}\lambda(\varrho)^{-1}J^{p}\left(\int_{\mathbb{R}^{d+1}}\left|h(s,z)\right|^{p}\,dzds\right),
\end{align*}
where
\[
J = \int_{0}^{T}\int_{\mathbb{R}^{d}}\left|\mathcal{S}_{\alpha,1,\phi}(\tau,y)\right|\,dyd\tau \leq C(\alpha,d,p,\delta_{0})\int_{0}^{T}\tau^{\alpha-1}\,d\tau \leq C(\alpha,d,p,\delta_{0},T),
\]
and we have used Proposition \ref{Pro.soltioun function}. Consequently, we obtain
\begin{align*}
\fint_{\mathcal{Q}_{\varrho}}\left|\mathcal{G}_{0}h(s,z)\right|^{p}\,dzds &\leq
C(\alpha,d,p,\delta_{0},T)\varrho^{-d}\lambda(\varrho)^{-1}\left(\int_{\mathbb{R}^{d+1}}\left|h(s,z)\right|^{p}\,dzds\right)\\
&\leq C(\alpha,d,p,\delta_{0},T)\fint_{\left(-3\lambda(\varrho),3\lambda(\varrho)\right)\times B_{3\varrho}}\left|h(s,z)\right|^{p}\,dzds\\
&\leq C(\alpha,d,p,\delta_{0},T)\mathcal{M}_{t,x}\left|h\right|^{p}(t,x).
\end{align*}

Applying Lemma \ref{independent estimate} with $p=q$, we derive
\begin{align*}
\fint_{\mathcal{Q}_{\varrho}}\left|\mathcal{G}_{1}h(s,z)\right|^{p}\,dzds
&\leq C(d)\lambda(\varrho)^{-1}\varrho^{-d}\int_{0}^{\infty}\int_{\mathbb{R}^{d}}\left|\mathcal{G}_{1}h(s-3\lambda(\varrho),z)\right|^{p}\,dzds\\
&\leq C(\alpha,d,p,\delta_{0})\lambda(\varrho)^{-1}\varrho^{-d}\int_{0}^{\infty}
\int_{\mathbb{R}^{d}}\left|h(s-3\lambda(\varrho),z)\right|^{p}\,dzds\\
&\leq C(\alpha,d,p,\delta_{0})\fint_{\left(-3\lambda(\varrho),3\lambda(\varrho)\right)\times B_{3\varrho}}\left|h(s,z)\right|^{p}\,dzds\\
&\leq C(\alpha,d,p,\delta_{0})\mathcal{M}_{t,x}\left|h\right|^{p}(t,x).
\end{align*}
\end{proof}
\begin{lemma}\label{supp estimate 2}
With $T<\infty$, $p\in(1,\infty)$, $\varrho>0$, take $h\in C_c^\infty(\mathbb{R}^{d+1})$ having support in $(-3\lambda(\varrho),\infty)\times\mathbb{R}^d$. Then for any $(t,x)\in\mathcal{Q}_{\varrho}$, the following estimates hold:
\begin{align}
\label{supp estimate 2.1}
\fint_{\mathcal{Q}_{\varrho}}\left|\mathcal{G}_{0}h(s,z)\right|^{p}\,dzds &\leq C(\alpha,d,p,\delta_{0},T)\mathcal{M}_{t,x}|h|^{p}(t,x),\\
\fint_{\mathcal{Q}_{\varrho}}\left|\mathcal{G}_{1}h(s,z)\right|^{p}\,dzds &\leq C(\alpha,d,p,\delta_{0})\mathcal{M}_{t,x}|h|^{p}(t,x).
\label{supp estimate 2.2}
\end{align}
\end{lemma}

\begin{proof}
Let $\xi\in C_{c}^{\infty}(\mathbb{R})$ be a cutoff function satisfying $\xi(t)=1$ for $|t|\leq 2\lambda(\varrho)$ and $\xi(t)=0$ for $|t|\geq 5\lambda(\varrho)/2$. From the definition of $\mathcal{G}_{k}$, we observe that $\mathcal{G}_{k}h=\mathcal{G}_{k}(\xi h)$ on $\mathcal{Q}_{\varrho}$. Without loss of generality, we may assume $h(t,x)=0$ for $|t|\geq 3\lambda(\varrho)$.

Let $\zeta\in C_{c}^{\infty}(\mathbb{R}^{d})$ be another cutoff function with $\zeta(x)=1$ on $B_{2\varrho}$ and $\zeta(x)=0$ outside $B_{5\varrho/2}$. Since $|h\zeta|\leq |h|$, we have the decomposition
$
\left|\mathcal{G}_{k}h\right|\leq \left|\mathcal{G}_{k}(\zeta h)\right|+\left|\mathcal{G}_{k}((1-\zeta)h)\right|.
$
The estimate for $\mathcal{G}_{k}(\zeta h)$ follows directly from Lemma \ref{supp estimate 1}. Therefore, we may further assume $h(t,x)=0$ for $x\in B_{2\varrho}$. For $(s,z)\in\mathcal{Q}_{\varrho}$, we consider two cases:
\begin{itemize}
\item If $\kappa<\varrho$, then $|z-y|<2\varrho$ for $y\in B_{\kappa}$, which implies $h(\tau,z-y)=0$.
\item If $\kappa\geq \varrho$, we have the inclusions
\[
|x-z|\leq 2\varrho,\quad B_{\varrho}(z)\subset B_{2\varrho+\kappa}(x)\subset B_{3\kappa}(x).
\]
\end{itemize}
Applying Proposition \ref{Pro.soltioun function} and the integration by parts formula \eqref{part intefration formula}, we obtain
\begin{align*}
&\fint_{\mathcal{Q}_{\varrho}}\left|\mathcal{G}_{1}h(s,z)\right|^{p}\,dzds\\
&=\fint_{\mathcal{Q}_{\varrho}}\left|\int_{-3\lambda(\varrho)}^{s}
\int_{\mathbb{R}^{d}}\mathcal{S}_{\alpha,1+\alpha,\phi}(s-\tau,y)h(\tau,z-y)\,dyd\tau\right|^{p}\,dzds\\
&\lesssim\fint_{\mathcal{Q}_{\varrho}}\left|\int_{-3\lambda(\varrho)}^{s}
|s-\tau|^{\alpha-1}\int_{|y|\geq \varrho}\frac{\phi(|y|^{-2})}{|y|^{d}}h(\tau,z-y)\,dyd\tau\right|^{p}\,dzds\\
&\lesssim\fint_{\mathcal{Q}_{\varrho}}\left|\int_{-3\lambda(\varrho)}^{s}|s-\tau|^{\alpha-1}
\int_{\varrho}^{\infty}\frac{d}{d\kappa}\left(\frac{\phi(\kappa^{-2})}{\kappa^{d}}\right)\int_{B_{3\kappa}(x)}h(\tau,y)\,dyd\kappa d\tau\right|^{p}\,dzds.
\end{align*}

From \eqref{B.S.T,func.bdd}, we derive that
\begin{align}\label{derivative relation 1}
\left|\frac{d}{d\kappa}\left(\frac{\phi(\kappa^{-2})}{\kappa^{d}}\right)\right|
=\left|\frac{-2\phi'(\kappa^{-2})\kappa^{d-3}-d\kappa^{d-1}\phi(\kappa^{-2})}{\kappa^{2d}}\right|\lesssim \frac{\phi(\kappa^{-2})}{\kappa^{d+1}}.
\end{align}

Combining \eqref{derivative relation 1} with H\"{o}lder's inequality yields
\begin{align*}
&\fint_{\mathcal{Q}_{\varrho}}\left|\mathcal{G}_{1}h(s,z)\right|^{p}\,dzds\\
&\lesssim\fint_{\mathcal{Q}_{\varrho}}\left|\int_{B_{3\kappa}(x)}\int_{-3\lambda(\varrho)}^{s}
\int_{\varrho}^{\infty}
\left(|s-\tau|^{\alpha-1}\frac{\phi(\kappa^{-2})}{\kappa^{d+1}}\right)^{\frac{1}{p}+\frac{1}{p'}}h(\tau,y)\,d\kappa d\tau dy\right|^{p}\,dzds\\
&\lesssim\fint_{\mathcal{Q}_{\varrho}}J(s,\varrho)^{p-1}\int_{-3\lambda(\varrho)}^{s}
|s-\tau|^{\alpha-1}\int_{\varrho}^{\infty}\frac{\phi(\kappa^{-2})}{\kappa^{d+1}}\int_{B_{3\kappa}(x)}
|h(\tau,y)|^{p}\,dyd\kappa d\tau dzds,
\end{align*}
where
\[
J(s,\varrho)=\int_{-3\lambda(\varrho)}^{s}|s-\tau|^{\alpha-1}\int_{\varrho}^{\infty}\frac{\phi(\kappa^{-2})}{\kappa}\,d\kappa d\tau.
\]

Using the change of variables and \eqref{Pro.Convergence}, we estimate
\[
|J(s,\varrho)|\leq\int_{0}^{3\lambda(\varrho)}|\tau|^{\alpha-1}\,d\tau
\int_{\varrho}^{\infty}\frac{\phi(\kappa^{-2})}{\kappa}\,d\kappa\lesssim \lambda(\varrho)^{\alpha}\phi(\varrho^{-2})\leq C(\alpha,d,p,\delta_{0}).
\]

By Fubini's theorem, we consequently obtain
\begin{align*}
&\fint_{\mathcal{Q}_{\varrho}}\left|\mathcal{G}_{1}h(s,z)\right|^{p}\,dzds\\
&\lesssim\lambda(\varrho)^{-1}\fint_{B_{\varrho}}\int_{-\lambda(\varrho)}^{0}\int_{-3\lambda(\varrho)}^{s}
|s-\tau|^{\alpha-1}\int_{\varrho}^{\infty}\frac{\phi(\kappa^{-2})}{\kappa^{d+1}}\int_{B_{3\kappa}(x)}
|h(\tau,y)|^{p}\,dyd\kappa d\tau dzds\\
&\lesssim\lambda(\varrho)^{-1}\fint_{B_{\varrho}}\int_{\varrho}^{\infty}\frac{\phi(\kappa^{-2})}{\kappa}
\int_{-3\lambda(\varrho)}^{0}\int_{\tau}^{0}|s-\tau|^{\alpha-1}\,ds \frac{1}{\kappa^{d}}\int_{B_{3\kappa}(x)}|h(\tau,y)|^{p}\,dyd\tau d\kappa dz\\
&\lesssim(\lambda(\varrho))^{\alpha}\fint_{B_{\varrho}}\int_{\varrho}^{\infty}
\frac{\phi(\kappa^{-2})}{\kappa}\left(\lambda(\varrho)^{-1}\kappa^{-d}\int_{-3\lambda(\varrho)}^{0}
\int_{B_{3\kappa}(x)}|h(\tau,y)|^{p}\,dyd\tau\right)\,d\kappa dz\\
&\leq C(\alpha,d,p,\delta_{0})(\lambda(\varrho))^{\alpha}\phi(\varrho^{-2})\mathcal{M}_{t,x}|h|^{p}(t,x)\\
&\leq C(\alpha,d,p,\delta_{0})\mathcal{M}_{t,x}|h|^{p}(t,x).
\end{align*}

For the $\mathcal{G}_{0}$ estimate, we proceed similarly:
\begin{align*}
&\fint_{\mathcal{Q}_{\varrho}}\left|\mathcal{G}_{0}h(s,y)\right|^{p}\,dzds\\
&=\fint_{\mathcal{Q}_{\varrho}}\left|\int_{-3\lambda(\varrho)}^{s}
\int_{\mathbb{R}^{d}}\mathbf{1}_{0<s-\tau<T}\mathcal{S}_{\alpha,1,\phi}(s-\tau,y)h(\tau,z-y)\,dyd\tau\right|^{p}\,dzdr\\
&\lesssim\fint_{\mathcal{Q}_{\varrho}}\left|\int_{-3\lambda(\varrho)}^{s}
\mathbf{1}_{0<s-\tau<T}|s-\tau|^{2\alpha-1}\int_{|y|\geq \varrho}\frac{\phi(|y|^{-2})}{|y|^{d}}h(\tau,z-y)\,dyd\tau\right|^{p}\,dzds\\
&\lesssim T^{\alpha}\fint_{\mathcal{Q}_{\varrho}}\left|\int_{-3\lambda(\varrho)}^{s}|s-\tau|^{\alpha-1}
\int_{\varrho}^{\infty}\frac{d}{d\kappa}\left(\frac{\phi(\kappa^{-2})}{\kappa^{d}}\right)\int_{B_{3\kappa}(x)}h(\tau,y)\,dyd\kappa d\tau\right|^{p}\,dzds.
\end{align*}

Repeating the argument used for $\mathcal{G}_{1}$ yields the final estimate \eqref{supp estimate 2.1}
\end{proof}

\begin{lemma}\label{supp estimate 3}
Let $T<\infty$, $1<p<\infty$, $\varrho>0$, and $h\in C_{c}^{\infty}\left(\mathbb{R}^{d+1}\right)$ with support contained in $\left(-\infty,-2\lambda(\varrho)\right)\times B_{3\varrho}$. Then for any $(t,x)\in\mathcal{Q}_{\varrho}$, the following estimates hold:
\begin{align}
\label{supp estimate 3.1}
\fint_{\mathcal{Q}_{\varrho}}\left|\mathcal{G}_{0}h(s,z)\right|^{p}\,dzds &\leq C(\alpha,d,p,\delta_{0},T)\mathcal{M}_{t,x}|h|^{p}(t,x),\\
\fint_{\mathcal{Q}_{\varrho}}\left|\mathcal{G}_{1}h(s,z)\right|^{p}\,dzds &\leq C(\alpha,d,p,\delta_{0})\mathcal{M}_{t,x}|h|^{p}(t,x).
\label{supp estimate 3.2}
\end{align}
\end{lemma}

\begin{proof}
Note that for $z\in B_{\varrho}$, $\left|y-z\right|\geq \left|y\right|-\left|z\right|$, we derive $h(r,y-z)=0$ for $y\in B_{4\varrho}^{c}$. Using the Minkowski inequality and H\"{o}lder's inequality, and noting that $\left|y-z\right|\leq 5\varrho$, we obtain
\begin{align*}
&\fint_{\mathcal{Q}_{\varrho}}\left|\mathcal{G}_{1}h(s,z)\right|^{p}\,dzds\\
&=\fint_{\mathcal{Q}_{\varrho}}
\left|\int_{-\infty}^{s}\int_{\mathbb{R}^{d}}\mathcal{S}_{\alpha,1+\alpha,\phi}(s-\tau,y)h(\tau,z-y)\,dyd\tau\right|^{p}\,dzds\\
&=\fint_{\mathcal{Q}_{\varrho}}\left|\int_{-\infty}^{s}\int_{B_{4\varrho}}
\mathcal{S}_{\alpha,1+\alpha,\phi}(s-\tau,y)h(\tau,z-y)\,dyd\tau\right|^{p}\,dzds\\
&\lesssim\lambda(\varrho)^{-1}\varrho^{-d}\int_{-\lambda(\varrho)}^{0}
\left(\int_{-\infty}^{-2\lambda(\varrho)}\int_{B_{4\varrho}}\left|\mathcal{S}_{\alpha,1+\alpha,\phi}(s-\tau,y)\right|
\left(\int_{B_{5\varrho}}\left|h(\tau,z)\right|^{p}\,dz\right)^{\frac{1}{p}}\,dyd\tau\right)^{p}ds\\
&\lesssim\lambda(\varrho)^{-1}\varrho^{-d}\int_{-\lambda(\varrho)}^{0}J(\varrho,s)^{p-1}\int_{\infty}^{-2\lambda(\varrho)}
\int_{B_{4\varrho}}\left|\mathcal{S}_{\alpha,1+\alpha,\phi}(s-\tau,y)\right|
\left(\int_{B_{5\varrho}}\left|h(\tau,z)\right|^{p}\,dz\right)dyd\tau ds,
\end{align*}
where
\begin{align*}
J(\varrho,s)=\int_{-\infty}^{-2\lambda(\varrho)}\int_{B_{4\varrho}}\left|\mathcal{S}_{\alpha,1+\alpha,\phi}(s-\tau,y)\right|\,dyd\tau.
\end{align*}
Observing that $\left(s-\tau\right)\in \left(\lambda(\varrho),\infty\right)$, we get
\begin{align*}
J(\varrho,s)&\lesssim
\int_{\lambda(\varrho)}^{\infty}\int_{B_{4\varrho}}\left|\mathcal{S}_{\alpha,1+\alpha,\phi}(\tau,y)\right|\,dyd\tau\\
&\lesssim\int_{\lambda(\varrho)}^{\lambda(4\varrho)}\int_{B_{4\varrho}}
\left|\mathcal{S}_{\alpha,1+\alpha,\phi}(\tau,y)\right|\,dyd\tau
+\int_{\lambda(4\varrho)}^{\infty}\int_{B_{4\varrho}}\left|\mathcal{S}_{\alpha,1+\alpha,\phi}(\tau,y)\right|\,dyd\tau.
\end{align*}
Since $\tau>\lambda(4\varrho)$, we derive $\tau^{\alpha}\phi(\varrho^{-2})\geq \tau^{\alpha}\phi(\varrho^{-2}/16)\geq 1$. By Proposition \ref{Pro.soltioun function},
\begin{align*}
\int_{\lambda(\varrho)}^{\lambda(4\varrho)}\int_{B_{4\varrho}}
\left|\mathcal{S}_{\alpha,1+\alpha,\phi}(\tau,y)\right|\,dyd\tau\lesssim\int_{\lambda(\varrho)}^{\lambda(4\varrho)}\tau^{-1}\,d\tau\leq C(\alpha,d,p,\delta_{0}),
\end{align*}
and
\begin{align*}
&\int_{\lambda(4\varrho)}^{\infty}\int_{B_{4\varrho}}\left|\mathcal{S}_{\alpha,1+\alpha,\phi}(\tau,y)\right|\,dyd\tau\\
&\lesssim\int_{\lambda(4\varrho)}^{\infty}\int_{B(4\varrho)}
\int_{[\phi(|y|^{-2})]^{-1}}^{2\tau^{\alpha}}\left(\phi^{-1}(r^{-1})\right)^{\frac{d}{2}}\tau^{-\alpha-1}\,drdyd\tau\\
&\lesssim\int_{\lambda(4\varrho)}^{\infty}\int_{B(4\varrho)}\left(\int_{[\phi(|y|^{-2})]^{-1}}^{[\phi(\varrho^{-2}/16)]^{-1}}
+\int_{[\phi(\varrho^{-2}/16)]^{-1}}^{2\tau^{\alpha}}
\right)\left(\phi^{-1}(r^{-1})\right)^{\frac{d}{2}}\tau^{-\alpha-1}\,drdyd\tau.
\end{align*}
The subsequent process repeats Kim \cite[Lemma 4.7]{Kim1} to obtain $\left|J(\varrho,s)\right|\leq C(\alpha,d,p,\delta_{0})$.

Hence we obtain
\begin{align*}
&\fint_{\mathcal{Q}_{\varrho}}\left|\mathcal{G}_{1}h(s,y)\right|^{p}\,dzds\\
&\lesssim\lambda(\varrho)^{-1}\varrho^{-d}\int_{-\lambda(\varrho)}^{0}\int_{-\infty}^{-2\lambda(\varrho)}
\int_{B_{4\varrho}}\left|\mathcal{S}_{\alpha,1+\alpha,\phi}(s-\tau,y)\right|dy\left(\int_{B_{5\varrho}}
\left|h(\tau,z)\right|^{p}\,dz\right)
d\tau ds\\
&\lesssim\lambda(\varrho)^{-1}\varrho^{-d}\int_{-\infty}^{-2\lambda(\varrho)}\int_{-\lambda(\varrho)-\tau}^{-\tau}
\int_{B_{4\varrho}}\left|\mathcal{S}_{\alpha,1+\alpha,\phi}(s,y)\right|\,dyds\left(\int_{B_{5\varrho}}
\left|h(\tau,z)\right|^{p}\,dz\right)d\tau\\
&\lesssim J_{1}+J_{2},
\end{align*}
where
\begin{align*}
J_{1}&=\lambda(\varrho)^{-1}\varrho^{-d}\int_{-\infty}^{-k\lambda(\varrho)}\int_{-\lambda(\varrho)-\tau}^{-\tau}
\int_{B_{4\varrho}}\left|\mathcal{S}_{\alpha,1+\alpha,\phi}(s,y)\right|\,dyds\left(\int_{B_{5\varrho}}
\left|h(\tau,z)\right|^{p}\,dz\right)d\tau,\\
J_{2}&=\lambda(\varrho)^{-1}\varrho^{-d}\int_{-k\lambda(\varrho)}^{-2\lambda(\varrho)}\int_{-\lambda(\varrho)-\tau}^{-\tau}
\int_{B_{4\varrho}}\left|\mathcal{S}_{\alpha,1+\alpha,\phi}(s,y)\right|\,dyds\left(\int_{B_{5\varrho}}
\left|h(\tau,z)\right|^{p}\,dz\right)d\tau,
\end{align*}
where $k\geq 3$ is chosen such that $(k-1)\lambda(\varrho)>\lambda(4\varrho)$, which clearly exists. Noting that $\tau\in\left(-\infty,-2\lambda(\varrho)\right)$ and using Proposition \ref{Pro.soltioun function}, we derive
\begin{align*}
J_{2}&=\lambda(\varrho)^{-1}\varrho^{-d}\int_{-k\lambda(\varrho)}^{-2\lambda(\varrho)}\int_{-\lambda(\varrho)-\tau}^{-\tau}
\int_{B_{4\varrho}}\left|\mathcal{S}_{\alpha,1+\alpha,\phi}(s,y)\right|\,dyds\left(\int_{B_{5\varrho}}
\left|h(\tau,z)\right|^{p}\,dz\right)d\tau\\
&\lesssim\lambda(\varrho)^{-1}\varrho^{-d}\int_{-k\lambda(\varrho)}^{-2\lambda(\varrho)}\int_{\lambda(\varrho)}^{k\lambda(\varrho)}
s^{-1}\,ds\left(\int_{B_{5\varrho}}
\left|h(\tau,z)\right|^{p}\,dz\right)d\tau\\
&\lesssim \log(k)\lambda(\varrho)^{-1}\varrho^{-d}\int_{-k\lambda(\varrho)}^{0}\int_{B_{5\varrho}}
\left|h(\tau,z)\right|^{p}\,dzd\tau\\
&\lesssim\mathcal{M}_{t,x}\left|h\right|^{p}(t,x).
\end{align*}
Moreover, since $-\lambda(\varrho)-\tau\geq \left(k-1\right)\lambda(\varrho)>\lambda(4\varrho)$, we derive
\begin{align*}
&\int_{-\lambda(\varrho)-\tau}^{-\tau}\int_{B_{4\varrho}}
\left|\mathcal{S}_{\alpha,1+\alpha,\phi}(s,y)\right|\left(\int_{B_{5\varrho}}\left|h(\tau,z)\right|^{p}\,dz\right)\,dyds\\
&\lesssim\int_{-\lambda(\varrho)-\tau}^{-\tau}\int_{B_{4\varrho}}\int_{\phi(|y|^{-2})^{-1}}^{2s^{\alpha}}s^{-\alpha-1}
\left(\phi^{-1}(r^{-1})\right)^{\frac{d}{2}}\left(\int_{B_{5\varrho}}\left|h(\tau,z)\right|^{p}\,dz\right)\,drdyds\\
&\lesssim\int_{-\lambda(\varrho)-\tau}^{-\tau}\left(\int_{0}^{[\phi(\varrho^{-2}/16)]^{-1}}\int_{\left|y\right|\leq [\phi^{-1}(r^{-1})]^{-\frac{1}{2}}}+\int_{[\phi(\varrho^{-2}/16)]^{-1}}^{2s^{\alpha}}
\int_{B_{4\varrho}}\right)s^{-\alpha-1}\left(\phi^{-1}(r^{-1})\right)^{\frac{d}{2}}\\
&\quad\cdot\left(\int_{B_{5\varrho}}\left|h(\tau,z)\right|^{p}\,dz\right)\,dydrds\\
&= I_{1}(\varrho,\tau)+I_{2}(\varrho,\tau).
\end{align*}
For $I_{1}$, we obtain
\begin{align*}
I_{1}(\varrho,\tau)&\lesssim[\phi(\varrho^{-2}/16)]^{-1}
\int_{-\lambda(\varrho)-\tau}^{-\tau}\int_{B_{5\varrho}}\left|h(\tau,z)\right|^{p}s^{-\alpha-1}\,dzds.
\end{align*}
Combining \eqref{lower scailing condition} and setting $m=\phi^{-1}(r^{-1})$, $M=\varrho^{-2}/16$, we obtain
\begin{align}\label{fangsuo condition}
\left(\phi^{-1}(r^{-1})\right)^{\frac{d}{2}}&=\left(\frac{\phi^{-1}(r^{-1})}{\varrho^{-2}/16}\varrho^{-2}/16\right)^{\frac{d}{2}}
\lesssim\varrho^{-d}r^{-\frac{d}{2}}[\phi(\varrho^{-2}/16)]^{-\frac{d}{2}}.
\end{align}
We derive
\begin{align*}
&I_{2}(\varrho,\tau)\\
&\lesssim\int_{-\lambda(\varrho)-\tau}^{-\tau}\int_{[\phi(\varrho^{-2}/16)]^{-1}}^{2s^{\alpha}}
\int_{B_{4\varrho}}\varrho^{-d}r^{-\frac{d}{2}}[\phi(\varrho^{-2}/16)]^{-\frac{d}{2}}s^{-\alpha-1}
\left(\int_{B_{5\varrho}}\left|h(\tau,z)\right|^{p}\,dz\right)\,dydrds\\
&\lesssim\int_{-\lambda(\varrho)-\tau}^{-\tau}[\phi(\varrho^{-2}/16)]^{-\frac{d}{2}}
\int_{[\phi(\varrho^{-2}/16)]^{-1}}^{2s^{\alpha}}r^{-\frac{d}{2}}s^{-\alpha-1}\,drds
\int_{B_{5\varrho}}\left|h(\tau,z)\right|^{p}\,dz\\
&\lesssim\int_{-\lambda(\varrho)-\tau}^{-\tau}[\phi(\varrho^{-2}/16)]^{-\frac{d}{2}}s^{-\alpha-1}
\left(s^{\frac{(2-d)\alpha}{2}}+[\phi(\varrho^{-2}/16)]^{-\frac{2-d}{2}}+\mathbf{1}_{d=2}\text{ }s^{\alpha\varepsilon}\phi(\varrho^{-2}/16)^{\varepsilon}\right)\\
&\quad\cdot\int_{B_{5\varrho}}\left|h(\tau,z)\right|^{p}\,dzds,\quad\text{ where }\varepsilon>0 \text{ is an arbitrary constant}.
\end{align*}
Now we obtain
\begin{align*}
J_{1}\lesssim \lambda(\varrho)^{-1}\varrho^{-d}\int_{-\infty}^{-k\lambda(\varrho)}I_{1}(\varrho,\tau)+I_{2}(\varrho,\tau)\,d\tau
\lesssim \lambda(\varrho)^{-1}\varrho^{-d}\left(J_{11}+J_{12}+J_{13}\right),
\end{align*}
where
\begin{align*}
J_{11}&=[\phi(\varrho^{-2}/16)]^{-1}\int_{-\infty}^{-k\lambda(\varrho)}
\int_{-\lambda(\varrho)-\tau}^{-\tau}\int_{B_{5\varrho}}
\left|h(\tau,z)\right|^{p}s^{-\alpha-1}
\,dzdsd\tau,\\
J_{12}&=[\phi(\varrho^{-2}/16)]^{-\frac{d}{2}}\int_{-\infty}^{-k\lambda(\varrho)}
\int_{-\lambda(\varrho)-\tau}^{-\tau}\int_{B_{5\varrho}}\left|h(\tau,z)\right|^{p}s^{-\alpha-1+\frac{\left(2-d\right)\alpha}{2}}
\,dzdsd\tau,\\
J_{13}&=\mathbf{1}_{d=2}[\phi(\varrho^{-2}/16)]^{-\frac{d}{2}+\varepsilon}
\int_{-\infty}^{-k\lambda(\varrho)}\int_{-\lambda(\varrho)-\tau}^{-\tau}\int_{B_{5\varrho}}\left|h(\tau,z)\right|^{p}
s^{-\alpha-1+\alpha\varepsilon}\,dzdsd\tau.
\end{align*}
We now estimate $J_{11}, J_{12}$, and $J_{13}$ separately. For $J_{11}$, we derive
\begin{align*}
J_{11}&\lesssim[\phi(\varrho^{-2}/16)]^{-1}\int_{-\infty}^{-k\lambda(\varrho)}\int_{B_{5\varrho}}
\left|h(\tau,z)\right|^{p}\,dz
\left(\left(-\lambda(\varrho)-\tau\right)^{-\alpha}-\left(-\tau\right)^{-\alpha}\right)\,d\tau\\
&\lesssim[\phi(\varrho^{-2}/16)]^{-1}\int_{-\infty}^{-k\lambda(\varrho)}\frac{d}{d\tau}\int_{\tau}^{0}\int_{B_{5\varrho}}
\left|h(\tilde{\tau},z)\right|^{p}\,dzd\tilde{\tau}
\left(\left(-\lambda(\varrho)-\tau\right)^{-\alpha}-\left(-\tau\right)^{-\alpha}\right)\,d\tau\\
&\lesssim[\phi(\varrho^{-2}/16)]^{-1}\lambda(\varrho)^{-\alpha}\int_{-k\lambda(\varrho)}^{0}\int_{B_{5\varrho}}
\left|h(\tilde{\tau},z)\right|^{p}\,dzd\tilde{\tau}  \\
&\quad+
[\phi(\varrho^{-2}/16)]^{-1}\int_{-\infty}^{-k\lambda(\varrho)}\left(\int_{\tau}^{0}\int_{B_{5\varrho}}
\left|h(\tilde{\tau},x)\right|^{p}\,dzd\tilde{\tau}\right)\left((-\lambda(\varrho)-\tau)^{-\alpha-1}-(-\tau)^{-\alpha-1}\right)\,d\tau\\
&\lesssim\lambda(\varrho)\varrho^{d}\mathcal{M}_{t,x}\left|h\right|^{p}(t,x).
\end{align*}
Similarly,
\begin{align*}
J_{12}&\lesssim [\phi(\varrho^{-2}/16)]^{-\frac{d}{2}}\lambda(\varrho)^{-\frac{d\alpha}{2}}\int_{-k\lambda(\varrho)}^{0}
\int_{B_{5\varrho}}\left|h(\tilde{\tau},z)\right|^{p}\,dzd\tilde{\tau}\\
&\lesssim\lambda(\varrho)\varrho^{d}\mathcal{M}_{t,x}\left|h\right|^{p}(t,x).
\end{align*}
Choosing $\varepsilon>0$ sufficiently small such that $\alpha\varepsilon<\alpha$, we derive
\begin{align*}
J_{13}&\lesssim\mathbf{1}_{d=2}[\phi(\varrho^{-2}/16)]^{-\frac{d}{2}+\varepsilon}\lambda(\varrho)^{-\alpha+\alpha\varepsilon}
\int_{-k\lambda(\varrho)}^{0}\int_{B_{5\varrho}}\left|h(\tilde{\tau},z)\right|^{p}\,dzd\tilde{\tau}\\
&\qquad+\mathbf{1}_{d=2}[\phi(\varrho^{-2}/16)]^{-\frac{d}{2}+\varepsilon}\int_{-\infty}^{-k\lambda(\varrho)}
\left(\int_{\tau}^{0}\int_{B_{5\varrho}}
\left|h(\tilde{\tau},x)\right|^{p}\,dzd\tilde{\tau}\right)\\
&\qquad\cdot\left((-\lambda(\varrho)-\tau)^{-\alpha-1+\alpha\varepsilon}-(-\tau)^{-\alpha-1+\alpha\varepsilon}\right)\,d\tau\\
&\lesssim\lambda(\varrho)\varrho^{d}\mathcal{M}_{t,x}\left|h\right|^{p}(t,x).
\end{align*}
In summary, we derive
\begin{align*}
\fint_{\mathcal{Q}_{\varrho}}\left|\mathcal{G}_{1}h(s,z)\right|^{p}\,dzds\lesssim J_{1}+J_{2}\lesssim\mathcal{M}_{t,x}\left|h\right|^{p}(t,x),
\end{align*}
where the constant $C$ depends on $\alpha,d,p,\delta_{0}$.

For $\mathcal{G}_{0}$, similar to $\mathcal{G}_{1}$, we can derive
\begin{align*}
&\fint_{\mathcal{Q}_{\varrho}}\left|\mathcal{G}_{0}h(s,z)\right|^{p}\,dzds\\
&\lesssim\fint_{\mathcal{Q}_{\varrho}}\left|\int_{-3\lambda(\varrho)}^{s}
\int_{\mathbb{R}^{d}}\mathbf{1}_{0<s-\tau<T}\mathcal{S}_{\alpha,1,\phi}(s-\tau,y)h(\tau,z-y)\,dyd\tau\right|^{p}\,dzds\\
&\lesssim\lambda(\varrho)^{-1}\varrho^{-d}\int_{-\lambda(\varrho)}^{0}J(\varrho,s)^{p-1}\int_{\infty}^{-2\lambda(\varrho)}
\int_{B_{4\varrho}}\mathbf{1}_{0<s-\tau<T}\left|\mathcal{S}_{\alpha,1,\phi}(s-\tau,y)\right|
\left(\int_{B_{5\varrho}}\left|h(\tau,z)\right|^{p}\,dz\right)dyd\tau ds,
\end{align*}
where
\begin{align*}
J(\varrho,s)=\int_{-\infty}^{-2\lambda(\varrho)}\int_{B_{4\varrho}}
\mathbf{1}_{0<s-\tau<T}\left|\mathcal{S}_{\alpha,1,\phi}(s-\tau,y)\right|\,dyd\tau.
\end{align*}
Combining Proposition \ref{Pro.soltioun function}, we obtain
\begin{align*}
J(\varrho,s)&\lesssim\int_{\lambda(\varrho)}^{\lambda(4\varrho)}\int_{B_{4\varrho}}
\mathbf{1}_{0<\tau<T}\left|\mathcal{S}_{\alpha,1,\phi}(\tau,y)\right|\,dyd\tau
+\int_{\lambda(4\varrho)}^{\infty}\int_{B_{4\varrho}}
\mathbf{1}_{0<\tau<T}\left|\mathcal{S}_{\alpha,1,\phi}(\tau,y)\right|\,dyd\tau\\
&\lesssim T^{\alpha}\int_{\lambda(\varrho)}^{\lambda(4\varrho)}\tau^{-1}\,d\tau\\
&\quad+T^{\alpha}
\int_{\lambda(4\varrho)}^{\infty}\int_{B(4\varrho)}\left(\int_{[\phi(|y|^{-2})]^{-1}}^{[\phi(\varrho^{-2}/16)]^{-1}}
+\int_{[\phi(\varrho^{-2}/16)]^{-1}}^{2\tau^{\alpha}}
\right)\left(\phi^{-1}(r^{-1})\right)^{\frac{d}{2}}\tau^{-\alpha-1}\,drdyd\tau\\
&\leq C(\alpha,p,d,\delta_{0},T).
\end{align*}
Hence, we obtain
\begin{align*}
&\fint_{\mathcal{Q}_{\varrho}}\left|\mathcal{G}_{0}h(s,z)\right|^{p}\,dzds\\
&\lesssim_{T}\lambda(\varrho)^{-1}\varrho^{-d}\int_{-\lambda(\varrho)}^{0}\int_{-\infty}^{-2\lambda(\varrho)}
\int_{B_{4\varrho}}\int_{B_{5\varrho}}\mathbf{1}_{0<s-\tau<T}\left|\mathcal{S}_{\alpha,1,\phi}(s-\tau,y)\right|
\left|h(\tau,z)\right|^{p}\,dzdyd\tau ds\\
&\lesssim J_{1}+J_{2},
\end{align*}
where
\begin{align*}
J_{1}&=\lambda(\varrho)^{-1}\varrho^{-d}\int_{-\infty}^{-k\lambda(\varrho)}\int_{-\lambda(\varrho)-\tau}^{-\tau}
\int_{B_{4\varrho}}\mathbf{1}_{0<s<T}\left|\mathcal{S}_{\alpha,1,\phi}(s,y)\right|\,dyds\left(\int_{B_{5\varrho}}
\left|h(\tau,z)\right|^{p}\,dz\right)d\tau,\\
J_{2}&=\lambda(\varrho)^{-1}\varrho^{-d}\int_{-k\lambda(\varrho)}^{-2\lambda(\varrho)}\int_{-\lambda(\varrho)-\tau}^{-\tau}
\int_{B_{4\varrho}}\mathbf{1}_{0<s<T}\left|\mathcal{S}_{\alpha,1,\phi}(s,y)\right|\,dyds\left(\int_{B_{5\varrho}}
\left|h(\tau,z)\right|^{p}\,dz\right)d\tau.
\end{align*}
Combining Proposition \ref{Pro.soltioun function} and following a similar process as for $\mathcal{G}_{1}$, we derive \eqref{supp estimate 3.1}.
\end{proof}
\begin{lemma}\label{supp estimate 4}
Let $T<\infty$, $1<p<\infty$, $\varrho>0$, and $h\in C_{c}^{\infty}\left(\mathbb{R}^{d+1}\right)$ with support contained in $\left(-\infty,-2\lambda(\varrho)\right)\times B^{c}_{2\varrho}$. Then for any $(t,x)\in\mathcal{Q}_{\varrho}$, the following estimates hold:
\begin{align}
\label{supp estimate 4.1}
\fint_{\mathcal{Q}_{\varrho}}\fint_{\mathcal{Q}_{\varrho}}
\left|\mathcal{G}_{0}h(s_{1},z_{1})-\mathcal{G}_{0}h(s_{2},z_{2})\right|^{p}\,dz_{1}dz_{2}ds_{1}ds_{2} &\lesssim_{T}
\mathcal{M}_{t}\mathcal{M}_{x}\left|h\right|^{p}(t,x)+\mathcal{M}_{t,x}\left|h\right|^{p}(t,x),\\
\fint_{\mathcal{Q}_{\varrho}}\fint_{\mathcal{Q}_{\varrho}}
\left|\mathcal{G}_{1}h(s_{1},z_{1})-\mathcal{G}_{1}h(s_{2},z_{2})\right|^{p}\,dz_{1}dz_{2}ds_{1}ds_{2} &\lesssim
\mathcal{M}_{t}\mathcal{M}_{x}\left|h\right|^{p}(t,x)+\mathcal{M}_{t,x}\left|h\right|^{p}(t,x).
\label{supp estimate 4.2}
\end{align}
\end{lemma}
\begin{proof}
First, we estimate $\mathcal{G}_{1}$. Note that
\begin{align*}
&\int_{\mathcal{Q}_{\varrho}}\int_{\mathcal{Q}_{\varrho}}
\left|\mathcal{G}_{1}h(s_{1},z_{1})-\mathcal{G}_{1}h(s_{2},z_{2})\right|^{p}\,dz_{1}dz_{2}ds_{1}ds_{2}\\
&\lesssim\int_{\mathcal{Q}_{\varrho}}\int_{\mathcal{Q}_{\varrho}}
\left|\mathcal{G}_{1}h(s_{1},z_{1})-\mathcal{G}_{1}h(s_{1},z_{2})\right|^{p}\,dz_{1}dz_{2}ds_{1}ds_{2}\\
&\quad+\int_{\mathcal{Q}_{\varrho}}\int_{\mathcal{Q}_{\varrho}}
\left|\mathcal{G}_{1}h(s_{1},z_{2})-\mathcal{G}_{1}h(s_{2},z_{2})\right|^{p}\,dz_{1}dz_{2}ds_{1}ds_{2}\\
&\triangleq:I_{1}+I_{2}.
\end{align*}
We estimate $I_{1}$ and $I_{2}$ separately. Note that $h(\tau,z)=0$ for $z\in B_{2\varrho}$. Using Minkowski's inequality and H\"{o}lder's inequality, we derive
\begin{align*}
&\int_{\mathcal{Q}_{\varrho}}\int_{\mathcal{Q}_{\varrho}}
\left|\mathcal{G}_{1}h(s_{1},z_{1})-\mathcal{G}_{1}h(s_{1},z_{2})\right|^{p}\,dz_{1}dz_{2}ds_{1}ds_{2}\\
&\lesssim\int_{\mathcal{Q}_{\varrho}}\int_{\mathcal{Q}_{\varrho}}
\left|\int_{-\infty}^{-2\lambda(\varrho)}\int_{\mathbb{R}^{d}}
\left(\mathcal{S}_{\alpha,1+\alpha,\phi}(s_{1}-\tau,z_{1}-y)
-\mathcal{S}_{\alpha,1+\alpha,\phi}(s_{1}-\tau,z_{2}-y)\right)h(\tau,y)dyd\tau\right|^{p}\\
&\qquad dz_{1}dz_{2}ds_{1}ds_{2}\\
&\lesssim\int_{\mathcal{Q}_{\varrho}}\int_{\mathcal{Q}_{\varrho}}J^{p-1}
\int_{-\infty}^{-2\lambda(\varrho)}\int_{B_{2\varrho}^{c}}
\left|\mathcal{S}_{\alpha,1+\alpha,\phi}(s_{1}-\tau,z_{1}-y)-\mathcal{S}_{\alpha,1+\alpha,\phi}(s_{1}-\tau,z_{2}-y)\right|
\left|h(\tau,y)\right|^{p}\,dyd\tau\\
&\qquad\,dz_{1}dz_{2}ds_{1}ds_{2},
\end{align*}
where
\begin{align*}
J=\int_{-\infty}^{-2\lambda(\varrho)}\int_{B_{2\varrho}^{c}}
\left|\mathcal{S}_{\alpha.1+\alpha,\phi}(s_{1}-\tau,z_{1}-y)-\mathcal{S}_{\alpha.1+\alpha,\phi}(s_{1}-\tau,z_{2}-y)\right|\,dyd\tau.
\end{align*}
By using Proposition \ref{Pro.soltioun function}, we derive that
\begin{align*}
J&\lesssim\int_{-\infty}^{-2\lambda(\varrho)}\int_{B_{2\varrho}^{c}}\int_{0}^{1}\left|\nabla\mathcal{S}_{\alpha,1+\alpha,\phi}
\left(s_{1}-\tau,\theta z_{1}+(1-\theta)z_{2}-y\right)\right|\left|z_{1}-z_{2}\right|\,d\theta dyd\tau
\end{align*}
Since $\theta\in (0,1)$ and $z_{1},z_{2}\in B_{\varrho}$, we have $\theta z_{1}+(1-\theta)z_{2}-y\in B_{\varrho}^{c}$, which implies
\begin{align*}
J&\lesssim \varrho\int_{-\infty}^{-2\lambda(\varrho)}\int_{B_{\varrho}^{c}}
\left|\nabla\mathcal{S}_{\alpha,1+\alpha,\phi}(s_{1}-\tau,y)\right|\,dyd\tau\\
&\lesssim\varrho\int_{\lambda(\varrho)}^{\infty}\int_{B_{\varrho}^{c}}
\left|\nabla\mathcal{S}_{\alpha,1+\alpha,\phi}(\tau,y)\right|\,dyd\tau\\
&\lesssim\varrho\int_{\lambda(\varrho)}^{\infty}\int_{\varrho}^{[\phi^{-1}(\tau^{-\alpha})]^{-\frac{1}{2}}}
\int_{\phi(l^{-2})^{-1}}^{2\tau^{\alpha}}\left(\phi^{-1}(r^{-1})\right)^{\frac{d+1}{2}}\tau^{-\alpha-1}l^{d-1}\,drdld\tau\\
&\quad+\varrho\int_{\lambda(\varrho)}^{\infty}\int_{[\phi^{-1}(\tau^{-\alpha})]^{-\frac{1}{2}}}^{\infty}\tau^{\alpha-1}
\frac{\phi(r^{-2})}{r^{2}}\,drd\tau.
\end{align*}
Following the steps in Kim \cite[Lemma 4.6]{Kim1}, we obtain $J\leq C\left(\alpha,p,d,\delta_{0}\right)$. Therefore,
\begin{align*}
I_{1}
&\lesssim\int_{\mathcal{Q}_{\varrho}}\int_{\mathcal{Q}_{\varrho}}
\int_{-\infty}^{-2\lambda(\varrho)}\int_{B_{2\varrho}^{c}}
\left|\mathcal{S}_{\alpha,1+\alpha,\phi}(s_{1}-\tau,z_{1}-y)-\mathcal{S}_{\alpha,1+\alpha,\phi}(s_{1}-\tau,z_{2}-y)\right|
\left|h(\tau,y)\right|^{p}\,dyd\tau\\
&\qquad\,dz_{1}dz_{2}ds_{1}ds_{2}\\
&\lesssim\varrho\int_{\mathcal{Q}_{\varrho}}\int_{\mathcal{Q}_{\varrho}}\int_{-\infty}^{-2\lambda(\varrho)}\int_{B_{2\varrho}^{c}}\int_{0}^{1}
\left|\nabla\mathcal{S}_{\alpha,1+\alpha,\phi}\left(s_{1}-\tau,\theta z_{1}+(1-\theta)z_{2}-y\right)\right|\left|h(\tau,y)\right|^{p}\,d\theta dyd\tau\\
&\qquad\,dz_{1}dz_{2}ds_{1}ds_{2}\\
&\lesssim\varrho\int_{\mathcal{Q}_{\varrho}}\int_{\mathcal{Q}_{\varrho}}\int_{-\infty}^{-2\lambda(\varrho)}
\int_{0}^{1}\int_{\varrho<\left|y\right|<[\phi^{-1}((s_{1}-\tau)^{-\alpha})]^{-\frac{1}{2}}}
\left|\nabla\mathcal{S}_{\alpha,1+\alpha,\phi}(s_{1}-\tau,y)\right|\left|h(\tau,\theta z_{1}+(1-\theta)z_{2}-y)\right|^{p}\\
&\qquad dyd\theta d\tau dz_{1}dz_{2}ds_{1}ds_{2}\\
&\quad +\varrho\int_{\mathcal{Q}_{\varrho}}\int_{\mathcal{Q}_{\varrho}}\int_{-\infty}^{-2\lambda(\varrho)}
\int_{0}^{1}\int_{\left|y\right|\geq[\phi^{-1}((s_{1}-\tau)^{-\alpha})]^{-\frac{1}{2}}}
\left|\nabla\mathcal{S}_{\alpha,1+\alpha,\phi}(s_{1}-\tau,y)\right|\left|h(\tau,\theta z_{1}+(1-\theta)z_{2}-y)\right|^{p}\\
&\qquad dyd\theta d\tau dz_{1}dz_{2}ds_{1}ds_{2}\\
&\triangleq: J_{11}+J_{12}.
\end{align*}
Let $\tilde{z}=\theta z_{1}+(1-\theta)z_{2}$. Since $\theta z_{1}+(1-\theta)z_{2}\in B_{2\varrho}$,
\begin{align*}
&\int_{0}^{1}\int_{\left|y\right|\geq[\phi^{-1}((s_{1}-\tau)^{-\alpha})]^{-\frac{1}{2}}}
\left|\nabla\mathcal{S}_{\alpha,1+\alpha,\phi}(s_{1}-\tau,y)\right|\left|h(\tau,\theta z_{1}+(1-\theta)z_{2}-y)\right|^{p}dyd\theta\\
&\lesssim\int_{0}^{1}\int_{\left|y\right|\geq[\phi^{-1}((s_{1}-\tau)^{-\alpha})]^{-\frac{1}{2}}}\left(s_{1}-\tau\right)^{\alpha-1}
\frac{\phi(\left|y\right|^{-2})}{\left|y\right|^{d+1}}\left|h(\tau,\theta z_{1}+(1-\theta)z_{2}-y)\right|^{p}\,dyd\theta\\
&\lesssim\int_{[\phi^{-1}((s_{1}-\tau)^{-\alpha})]^{-\frac{1}{2}}}^{\infty}
\left(s_{1}-\tau\right)^{\alpha-1}\frac{\phi(|\kappa|^{-2})}{|\kappa|^{d+2}}\int_{B_{3\kappa}(x)}\left|h(\tau,z)\right|^{p}\,dzd\kappa.
\end{align*}
Note that $s_{1}-\tau\in(\lambda(\varrho),\infty)$, and set $m=\phi^{-1}((s_{1}-\tau)^{-\alpha})$, $M=\varrho^{-2}$. Using \eqref{lower scailing condition}, we obtain
$$
[\phi^{-1}((s_{1}-\tau)^{-\alpha})]^{\frac{1}{2}}\lesssim \phi(\varrho^{-2})^{-\frac{1}{2}}\varrho^{-1}\left(s_{1}-\tau\right)^{-\frac{\alpha}{2}}.
$$
Hence we obtain that
\begin{align*}
J_{12}&\lesssim\varrho\int_{\mathcal{Q}_{\varrho}}\int_{\mathcal{Q}_{\varrho}}\int_{-\infty}^{-2\lambda(\varrho)}
\int_{[\phi^{-1}((s_{1}-\tau)^{-\alpha})]^{-\frac{1}{2}}}^{\infty}
\left(s_{1}-\tau\right)^{\alpha-1}\frac{\phi(|\kappa|^{-2})}{|\kappa|^{d+2}}
\int_{B_{3\kappa}(x)}\left|h(\tau,z)\right|^{p}\,dzd\kappa d\tau\\
&\quad\,dz_{1}dz_{2}ds_{1}ds_{2}\\
&\lesssim\varrho\lambda(\varrho)\varrho^{2d}\int_{-\infty}^{-2\lambda(\varrho)}\int_{-\lambda(\varrho)}^{0}
\int_{[\phi^{-1}((s_{1}-\tau)^{-\alpha})]^{-\frac{1}{2}}}^{\infty}
\left(s_{1}-\tau\right)^{\alpha-1}\frac{\phi(|\kappa|^{-2})}{|\kappa|^{d+2}}
\int_{B_{3\kappa}(x)}\left|h(\tau,z)\right|^{p}\\
&\quad\,dzd\kappa ds_{1}d\tau\\
&\lesssim\phi(\varrho^{-2})^{-\frac{1}{2}}\lambda(\varrho)\varrho^{2d}\int_{-\infty}^{-2\lambda(\varrho)}\int_{-\lambda(\varrho)}^{0}
\int_{[\phi^{-1}((s_{1}-\tau)^{-\alpha})]^{-\frac{1}{2}}}^{\infty}\left(s_{1}-\tau\right)^{\frac{\alpha}{2}-1}
\frac{\phi(\left|\kappa\right|^{-2})}{\left|\kappa\right|}\mathcal{M}_{x}\left|h\right|^{p}(\tau,x)\\
&\quad\,d\kappa ds_{1}d\tau\\
&\lesssim\phi(\varrho^{-2})^{-\frac{1}{2}}\lambda(\varrho)\varrho^{2d}\int_{-\infty}^{-2\lambda(\varrho)}
\int_{-\lambda(\varrho)-\tau}^{-\tau}s^{-\frac{\alpha}{2}-1}\mathcal{M}_{x}\left|h\right|^{p}(\tau,x)\,dsd\tau\\
&\lesssim\phi(\varrho^{-2})^{-\frac{1}{2}}\lambda(\varrho)\varrho^{2d}\left(\lambda(\varrho)\right)^{-\frac{\alpha}{2}}
\int_{-2\lambda(\varrho)}^{0}\mathcal{M}_{x}\left|h\right|^{p}(\tau,x)\,d\tau\\
&\quad+\phi(\varrho^{-2})^{-\frac{1}{2}}\lambda(\varrho)\varrho^{2d}\int_{-\infty}^{-2\lambda(\varrho)}
\int_{\tau}^{0}\mathcal{M}_{x}\left|h\right|^{p}(\tilde{\tau},x)\,d\tilde{\tau}
\left((-\lambda(\varrho)-\tau)^{-\frac{\alpha}{2}-1}-(-\tau)^{-\frac{\alpha}{2}-1}\right)\,d\tau\\
&\lesssim\left(\lambda(\varrho)\varrho^{d}\right)^{2}\mathcal{M}_{t}\mathcal{M}_{x}\left|h\right|^{p}(t,x).
\end{align*}
For $J_{11}$,~by using \eqref{part intefration formula}, we derive that
\begin{align*}
&\int_{0}^{1}\int_{\varrho<\left|y\right|<[\phi^{-1}((s_{1}-\tau)^{-\alpha})]^{-\frac{1}{2}}}
\left|\nabla\mathcal{S}_{\alpha,1+\alpha,\phi}(s_{1}-\tau,y)\right|\left|h(\tau,\tilde{z}-y)\right|^{p}\,dyd\theta\\
&\lesssim\left(s_{1}-\tau\right)^{-\alpha-1}
\int_{0}^{1}\int_{\varrho<\left|y\right|<[\phi^{-1}((s_{1}-\tau)^{-\alpha})]^{-\frac{1}{2}}}
\int_{(\phi(|y|^{-2}))^{-1}}^{2(s_{1}-\tau)^{\alpha}}\left[\phi^{-1}(r^{-1})\right]^{\frac{d+1}{2}}
\left|h(\tau,\tilde{z}-y)\right|^{p}\,drdyd\theta\\
&\lesssim\left(s_{1}-\tau\right)^{-\alpha-1}\int_{\varrho}^{[\phi^{-1}((s_{1}-\tau)^{-\alpha})]^{-\frac{1}{2}}}
\frac{d}{d\kappa}\int_{\phi(\left|\kappa\right|^{-2})^{-1}}^{2(s_{1}-\tau)^{\alpha}}\left[\phi^{-1}(r^{-1})\right]^{\frac{d+1}{2}}\,dr
\int_{B_{3\kappa}(x)}\left|h(\tau,z)\right|^{p}\,dzd\kappa\\
&\quad+\left(s_{1}-\tau\right)^{-\alpha-1}\int_{(s_{1}-\tau)^{\alpha}}^{2(s_{1}-\tau)^{\alpha}}
\left[\phi^{-1}(r^{-1})\right]^{\frac{d+1}{2}}\,dr
\int_{B_{[\phi^{-1}((s_{1}-\tau)^{-\alpha})]^{-\frac{1}{2}}}}\left|h(\tau,\tilde{z}-y)\right|^{p}
\,dy\\
&\triangleq:J_{111}+J_{112}.
\end{align*}
For $J_{112}$, note that:
\begin{align*}
&\left[\phi^{-1}(r^{-1})\right]^{\frac{d+1}{2}}\lesssim
r^{-\frac{d+1}{2}}\left[\phi((s_{1}-\tau)^{-\alpha})\right]^{-\frac{d+1}{2}}\left(s_{1}-\tau\right)^{-\frac{\alpha(d+1)}{2}},\text{ }\varrho<[\phi^{-1}((s_{1}-\tau)^{-\alpha})]^{-\frac{1}{2}},\\
&\left[\phi^{-1}(r^{-1})\right]^{\frac{1}{2}}
\lesssim \left(\phi(\varrho^{-2})\right)^{-\frac{1}{2}}r^{-\frac{1}{2}}\varrho^{-1},
\qquad\left(\lambda(\varrho)\right)^{\alpha}<\left(s_{1}-\tau\right)^{\alpha}<r,\qquad
\end{align*}
we derive:
\begin{align*}
J_{112}&\lesssim\left(s_{1}-\tau\right)^{-\alpha-1}\int_{(s_{1}-\tau)^{\alpha}}^{2(s_{1}-\tau)^{\alpha}}
\left[\phi^{-1}(r^{-1})\right]^{\frac{1}{2}}\,dr\mathcal{M}_{x}\left|h\right|^{p}(\tau,x)\\
&\lesssim\left(s_{1}-\tau\right)^{-\alpha-1}
\left(\phi(\varrho^{-2})\right)^{-\frac{1}{2}}\varrho^{-1}\int_{(\lambda(\varrho))^{\alpha}}^{2(s_{1}-\tau)^{\alpha}}
r^{-\frac{1}{2}}\,dr\mathcal{M}_{x}\left|h\right|^{p}(\tau,x)
\end{align*}
For $J_{111}$, observe that:
\begin{align*}
\left(s_{1}-\tau\right)^{\alpha}\phi(\left|\kappa\right|^{-2})\geq 1,\quad\left|\kappa\right|^{-2}
\leq\phi^{-1}\left(2^{-1}r^{-1}\right)\lesssim\phi^{-1}(r^{-1})
\end{align*}
\begin{align*}
\left|\frac{d}{d\kappa}\int_{\phi(\left|\kappa\right|^{-2})^{-1}}^{2(s_{1}-\tau)^{\alpha}}
\left[\phi^{-1}(r^{-1})\right]^{\frac{d+1}{2}}\,dr\right|
&\lesssim\left|\left[\phi^{-1}(\zeta^{-1})\right]^{\frac{d+1}{2}}
\big|_{\zeta=\phi\left(\left|\kappa\right|^{-2}\right)^{-1}}\phi(\left|\kappa\right|^{-2})^{-2}\phi'(\left|\kappa\right|^{-2})
\left|\kappa\right|^{-3}\right|\\
&=\left|\frac{\phi(\left|\kappa\right|^{-2})^{-1}}{\left|\kappa\right|^{\frac{d+1}{2}}}\right|\\
&\lesssim\left|\int_{\phi(\left|\kappa\right|^{-2})^{-1}}^{2\phi(\left|\kappa\right|^{-2})^{-1}}
\left[\phi^{-1}(r^{-1})\right]^{\frac{d+2}{2}}\,dr\right|\\
&\lesssim\left|\int_{\phi(\left|\kappa\right|^{-2})^{-1}}^{2(s_{1}-\tau)^{\alpha}}
\left[\phi^{-1}(r^{-1})\right]^{\frac{d+2}{2}}\,dr\right|,
\end{align*}
we obtain:
\begin{align*}
J_{111}&\lesssim\left(s_{1}-\tau\right)^{-\alpha-1}
\int_{\varrho}^{\left[\phi^{-1}((s_{1}-\tau)^{-\alpha})\right]^{-\frac{1}{2}}}
\int_{\phi(\left|\kappa\right|^{-2})^{-1}}^{2(s_{1}-\tau)^{\alpha}}\int_{B_{3\kappa}(x)}
\left[\phi^{-1}(r^{-1})\right]^{\frac{d+2}{2}}\left|h(\tau,z)\right|^{p}\,dzdrd\kappa\\
&=\left(s_{1}-\tau\right)^{-\alpha-1}\int_{\phi(\left|\varrho\right|^{-2})^{-1}}^{2(s_{1}-\tau)^{\alpha}}
\int_{\varrho}^{\left[\phi^{-1}(r^{-1})\right]^{-\frac{1}{2}}}\left[\phi^{-1}(r^{-1})\right]^{\frac{d+2}{2}}
\kappa^{d}\left(\frac{1}{\kappa^{d}}\int_{B_{3\kappa}(x)}\left|h(\tau,z)\right|^{p}\,dz\right)\,d\kappa dr\\
&\lesssim\left(s_{1}-\tau\right)^{-\alpha-1}\int_{\phi(\left|\varrho\right|^{-2})^{-1}}^{2(s_{1}-\tau)^{\alpha}}
\int_{\varrho}^{\left[\phi^{-1}(r^{-1})\right]^{-\frac{1}{2}}}\left[\phi^{-1}(r^{-1})\right]^{\frac{d+2}{2}}
\kappa^{d}\left(\frac{1}{\kappa^{d}}\int_{B_{3\kappa}(x)}\left|h(\tau,z)\right|^{p}\,dz\right)\,d\kappa dr\\
&\lesssim\left(s_{1}-\tau\right)^{-\alpha-1}\int_{\phi(\left|\varrho\right|^{-2})^{-1}}^{2(s_{1}-\tau)^{\alpha}}
\left[\phi^{-1}(r^{-1})\right]^{\frac{1}{2}}\,dr\mathcal{M}_{x}\left|h\right|^{p}(\tau,x)\\
&\lesssim\left(s_{1}-\tau\right)^{-\alpha-1}\left(\phi(\varrho^{-2})\right)^{-\frac{1}{2}}\varrho^{-1}
\int_{\phi(\left|\varrho\right|^{-2})^{-1}}^{2(s_{1}-\tau)^{\alpha}}
r^{-\frac{1}{2}}\,dr\mathcal{M}_{x}\left|h\right|^{p}(\tau,x),
\end{align*}
thus,
\begin{align*}
J_{11}&\lesssim\left(\phi(\varrho^{-2})\right)^{-\frac{1}{2}}\varrho^{-1}\varrho
\int_{\mathcal{Q}_{\varrho}}\int_{\mathcal{Q}_{\varrho}}\int_{-\infty}^{-2\lambda(\varrho)}
\left(s_{1}-\tau\right)^{-\alpha-1}
\int_{\phi(\left|\varrho\right|^{-2})^{-1}}^{2(s_{1}-\tau)^{\alpha}}
r^{-\frac{1}{2}}\,dr\mathcal{M}_{x}\left|h\right|^{p}(\tau,x)\,d\tau\\
&\quad\,dz_{1}dz_{2}ds_{1}ds_{2}\\
&\lesssim\left(\phi(\varrho^{-2})\right)^{-\frac{1}{2}}\lambda(\varrho)\varrho^{2d}\int_{-\infty}^{-2\lambda(\varrho)}
\int_{-\lambda(\varrho)-\tau}^{-\tau}\left(s\right)^{-\alpha-1}\mathcal{M}_{x}\left|h\right|^{p}(\tau,x)\,dsd\tau\\
&\lesssim\left(\lambda(\varrho)\varrho^{d}\right)^{2}\mathcal{M}_{t}\mathcal{M}_{x}\left|h\right|^{p}(t,x).
\end{align*}
Now, we derive
\begin{align*}
\fint_{\mathcal{Q}_{\varrho}}\fint_{\mathcal{Q}_{\varrho}}
\left|\mathcal{G}_{1}h(s_{1},z_{1})-\mathcal{G}_{1}h(s_{1},z_{2})\right|^{p}\,dz_{1}dz_{2}ds_{1}ds_{2}\lesssim
\mathcal{M}_{t}\mathcal{M}_{x}\left|h\right|^{p}(t,x).
\end{align*}

Next we estimate $I_{2}$. Using H\"{o}lder's inequality:
\begin{align*}
I_{2}&\lesssim\int_{\mathcal{Q}_{\varrho}}\int_{\mathcal{Q}_{\varrho}}J^{p-1}\int_{-\infty}^{-2\lambda(\varrho)}\int_{\mathbb{R}^{d}}
\left|\mathcal{S}_{\alpha,1+\alpha,\phi}(s_{1}-\tau,z_{2}-y)-\mathcal{S}_{\alpha,1+\alpha,\phi}(s_{2}-\tau,z_{2}-y)\right|
\left|h(\tau,y)\right|^{p}\\
&\quad\,dyd\tau dz_{1}dz_{2}ds_{1}ds_{2},
\end{align*}
where
\begin{align*}
J=\int_{-\infty}^{-2\lambda(\varrho)}\int_{\mathbb{R}^{d}}
\left|\mathcal{S}_{\alpha,1+\alpha,\phi}(s_{1}-\tau,z_{2}-y)-\mathcal{S}_{\alpha,1+\alpha,\phi}(s_{2}-\tau,z_{2}-y)\right|\,dyd\tau.
\end{align*}
By Proposition \ref{Pro.soltioun function} and without loss of generality assuming $s_{1}>s_{2}$:
\begin{align*}
J&\lesssim\int_{-\infty}^{-2\lambda(\varrho)}\int_{\mathbb{R}^{d}}\int_{s_{1}}^{s_{2}}
\left|\mathcal{S}_{\alpha,2+\alpha,\phi}(s-\tau,z_{2}-y)\right|\,dsdyd\tau\\
&=\int_{s_{1}}^{s_{2}}\int_{-\infty}^{-2\lambda(\varrho)}\left(s-\tau\right)^{-2}\,d\tau ds\\
&\lesssim \frac{1}{\lambda(\varrho)}\left(s_{1}-s_{2}\right)\\
&\leq C(\alpha,d,p,\delta_{0}).
\end{align*}
Let $\tilde{s}=\theta s_{1}+(1-\theta)s_{2}$ for any $0<\theta<1$. Note that $z_{2}-y\in B_{2\varrho}$ for $y\in B_{\varrho}$, $\tilde{s}-\tau\sim s_{1}-\tau$ for $s_{1},s_{2}\in (-\lambda(\varrho),0)$, $\tau\in(-\infty,-2\lambda(\varrho))$, and
\begin{align*}
B_{\kappa}(y)\subset B_{2\kappa+\varrho}(x)\subset B_{3\kappa}(x) \text{ for }\kappa\geq \varrho, \left|y\right|\leq \kappa,
\end{align*}
thus, combine \eqref{Pro.Convergence} we obtain:
\begin{align*}
I_{2}&\lesssim\int_{\mathcal{Q}_{\varrho}}\int_{\mathcal{Q}_{\varrho}}\int_{-\infty}^{-2\lambda(\varrho)}
\int_{\mathbb{R}^{d}}\left|s_{1}-s_{2}\right|
\left|\mathcal{S}_{\alpha,2+\alpha,\phi}(\tilde{s}-\tau,y)\right|\left|h(\tau,z_{2}-y)\right|^{p}
\,dyd\tau dz_{1}dz_{2}ds_{1}ds_{2}\\
&\lesssim\lambda(\varrho)\int_{\mathcal{Q}_{\varrho}}\int_{\mathcal{Q}_{\varrho}}\int_{-\infty}^{-2\lambda(\varrho)}
\int_{B_{\varrho}^{c}}\left|\tilde{s}-\tau\right|^{\alpha-2}\frac{\phi(|y|^{-2})}{|y|^{d}}\left|h(\tau,z_{2}-y)\right|^{p}
\,dyd\tau dz_{1}dz_{2}ds_{1}ds_{2}\\
&\lesssim \lambda(\varrho)\int_{\mathcal{Q}_{\varrho}}\int_{\mathcal{Q}_{\varrho}}\int_{-\infty}^{-2\lambda(\varrho)}
\left|\tilde{s}-\tau\right|^{\alpha-2}
\int_{\varrho}^{\infty}\frac{\phi(|\kappa|^{-2})}{\left|\kappa\right|^{d+1}}\int_{B_{3\kappa}(x)}\left|h(\tau,z)\right|^{p}\,dzd\kappa d\tau dz_{1}dz_{2}ds_{1}ds_{2}\\
&\lesssim\left(\lambda(\varrho)\varrho^{d}\right)^{2}\int_{-\infty}^{-2\lambda(\varrho)}\int_{-\lambda(\varrho)-\tau}^{-\tau}
s_{1}^{\alpha-2}\int_{\varrho}^{\infty}\frac{\phi(|\kappa|^{-2})}{|\kappa|}\frac{1}{\left|\kappa\right|^{d}}
\int_{B_{3\kappa}(x)}\left|h(\tau,z)\right|^{p}\,dzd\kappa ds_{1}d\tau\\
&\lesssim\left(\lambda(\varrho)\varrho^{d}\right)^{2}(\lambda(\varrho))^{\alpha-1}\int_{\varrho}^{\infty}
\frac{\phi(|\kappa|^{-2})}{|\kappa|}\int_{-2\lambda(\varrho)}^{0}\frac{1}{\left|\kappa\right|^{d}}
\int_{B_{3\kappa}(x)}\left|h(\tilde{\tau},z)\right|^{p}\,dzd\tilde{\tau} d\kappa\\
&\quad +\left(\lambda(\varrho)\varrho^{d}\right)^{2}\int_{-\infty}^{-2\lambda(\varrho)}\int_{\varrho}^{\infty}
\frac{\phi(|\kappa|^{-2})}{|\kappa|}\int_{\tau}^{0}\frac{1}{|\kappa|^{d}}
\int_{B_{3\kappa}(x)}\left|h(\tilde{\tau},z)\right|^{p}\,dzd\tilde{\tau} d\kappa\\
&\qquad\left((-\lambda(\varrho)-\tau)^{\alpha-2}-(-\tau)^{\alpha-2}\right)\,d\tau\\
&\lesssim\left(\lambda(\varrho)\varrho^{d}\right)^{2}\mathcal{M}_{t,x}\left|h\right|^{p}(t,x).
\end{align*}
In conclusion, we obtain:
\begin{align*}
&\fint_{\mathcal{Q}_{\varrho}}\fint_{\mathcal{Q}_{\varrho}}
|\mathcal{G}_{1}h(s_{1},z_{1})-\mathcal{G}_{1}h(s_{2},z_{2})|^{p}\,dz_{1}dz_{2}ds_{1}ds_{2}\\
&\lesssim\frac{1}{\lambda(\varrho)^{2}\varrho^{2d}}(I_{1}+I_{2})\lesssim\mathcal{M}_{t,x}|h|^{p}(t,x)
+\mathcal{M}_{t}\mathcal{M}_{x}|h|^{p}(t,x),
\end{align*}
where the constant depends on $\alpha,p,d,\delta_{0}$.

For $\mathcal{G}_{0}$, the estimate is similar to $\mathcal{G}_{1}$. Following the same steps as for $\mathcal{G}_{1}$, but now using the estimates for $\mathcal{S}_{\alpha,1,\phi}$ from Proposition \ref{Pro.soltioun function}, the constant depends on $\alpha,p,d,\delta_{0},T$.
\end{proof}
\begin{lemma}\label{supp estimate 5}
Let $T<\infty$, $\varrho>0$, $1<p<\infty$, $(t_{0},x_{0})\in\mathbb{R}\times\mathbb{R}^{d}$, and $h\in C_{c}^{\infty}\left(\mathbb{R}\times\mathbb{R}^{d}\right)$. Then for any $(t,x)\in\mathcal{Q}_{\varrho}(t_{0},x_{0})$, the following estimates hold:
\begin{align}
\label{supp estimate 5.1}
\fint_{\mathcal{Q}_{\varrho}(t_{0},x_{0})}\left|\mathcal{G}_{0}h(s,z)-\left(\mathcal{G}_{0}h\right)
_{\mathcal{Q}_{\varrho}(t_{0},x_{0})}\right|^{p}\,dzds&\lesssim_{T}\mathcal{M}_{t,x}\left|h\right|^{p}(t,x)
+\mathcal{M}_{t}\mathcal{M}_{x}\left|h\right|^{p}(t,x),\\
\fint_{\mathcal{Q}_{\varrho}(t_{0},x_{0})}\left|\mathcal{G}_{1}h(s,z)-\left(\mathcal{G}_{1}h\right)
_{\mathcal{Q}_{\varrho}(t_{0},x_{0})}\right|^{p}\,dzds&\lesssim\mathcal{M}_{t,x}\left|h\right|^{p}(t,x)
+\mathcal{M}_{t}\mathcal{M}_{x}\left|h\right|^{p}(t,x),
\label{supp estimate 5.2}\\
\left(\mathcal{G}_{0}h\right)^{\sharp}(t,x)&\lesssim_{T}\left(\mathcal{M}_{t,x}\left|h\right|^{p}(t,x)\right)^{\frac{1}{p}}
+\left(\mathcal{M}_{t}\mathcal{M}_{x}\left|h\right|^{p}(t,x)\right)^{\frac{1}{p}},
\label{supp estimate 5.3}\\
\left(\mathcal{G}_{1}h\right)^{\sharp}(t,x)&\lesssim\left(\mathcal{M}_{t,x}\left|h\right|^{p}(t,x)\right)^{\frac{1}{p}}
+\left(\mathcal{M}_{t}\mathcal{M}_{x}\left|h\right|^{p}(t,x)\right)^{\frac{1}{p}}.
\label{supp estimate 5.4}
\end{align}
\end{lemma}

\begin{proof}
It suffices to prove \eqref{supp estimate 5.2} and \eqref{supp estimate 5.4}, as the proofs of \eqref{supp estimate 5.1}
and \eqref{supp estimate 5.3} are analogous.

First observe that
\begin{align*}
\mathcal{G}_{1}h(t+t_{0},x+x_{0})&=\int_{-\infty}^{t+t_{0}}\int_{\mathbb{R}^{d}}\mathcal{S}_{\alpha,1+\alpha,\phi}
\left(t+t_{0}-\tau,x+x_{0}-z\right)h(\tau,z)\,dzd\tau\\
&=\int_{-\infty}^{t}\int_{\mathbb{R}^{d}}\mathcal{S}_{\alpha,1+\alpha,\phi}
\left(t-\tau,x-z\right)\tilde{h}(\tau,z)\,dzd\tau,
\end{align*}
where $\tilde{h}(t,x)=h\left(t+t_{0},x+x_{0}\right)$. Consequently, we have:
\begin{align*}
\fint_{\mathcal{Q}_{\varrho}(t_{0},x_{0})}\left|\mathcal{G}_{1}h(s,z)-\left(\mathcal{G}_{1}h\right)
_{\mathcal{Q}_{\varrho}(t_{0},x_{0})}\right|\,dzds&=\fint_{\mathcal{Q}_{\varrho}(0,0)}
\left|\mathcal{G}_{1}\tilde{h}(s,z)-\left(\mathcal{G}_{1}\tilde{h}\right)
_{\mathcal{Q}_{\varrho}(0,0)}\right|\,dzds.
\end{align*}

Without loss of generality, we may assume $(t_{0},x_{0})=\left(0,0\right)$. Let $\zeta\in C^{\infty}\left(\mathbb{R}^{d}\right)$ and $\eta\in C^{\infty}\left(\mathbb{R}\right)$ be cutoff functions satisfying $0\leq \zeta\leq 1$, $0\leq\eta\leq 1$ with
\begin{align*}
\zeta=
\begin{cases}
1 & \text{on } B_{\frac{7\varrho}{3}}\\
0 & \text{on } B^{c}_{\frac{8\varrho}{3}}
\end{cases},\quad
\eta=
\begin{cases}
1 & \text{on } (\frac{-7\lambda(\varrho)}{3},\infty)\\
0 & \text{on } (-\infty,\frac{-8\lambda(\varrho)}{3})
\end{cases}.
\end{align*}

We decompose $h$ as follows:
$$
h=h\eta+h(1-\eta)\zeta+h(1-\eta)(1-\zeta)\triangleq h_{1}+h_{2}+h_{3}.
$$

By the linearity of $\mathcal{G}_{1}$ and the fact that $|h_{i}|\leq |h|$ for $i=1,2,3$, we obtain
\begin{align*}
\left|\mathcal{G}_{1}h-\left(\mathcal{G}_{1}h\right)_{\mathcal{Q}_{\varrho}}\right|\leq\sum_{i=1}^{3}
\left|\mathcal{G}_{1}h_{i}-\left(\mathcal{G}_{1}h_{i}\right)_{\mathcal{Q}_{\varrho}}\right|.
\end{align*}

Applying Lemmas \ref{supp estimate 1}-\ref{supp estimate 4} to each term yields the estimate \eqref{supp estimate 5.2}.

Moreover, using H\"{o}lder's inequality, we derive the sharp function estimate:
\begin{align*}
\left(\mathcal{G}_{1}h\right)^{\sharp}(t,x)&=\sup_{\substack{(t,x)\in\mathcal{Q}_{\varrho}\\ \varrho>0}}\fint_{\mathcal{Q}_{\varrho}}
\left|\mathcal{G}_{1}h(s,z)-\left(\mathcal{G}_{1}h\right)_{\mathcal{Q}_{\varrho}}\right|\,dzds\\
&\lesssim\sup_{\substack{(t,x)\in\mathcal{Q}_{\varrho}\\ \varrho>0}}
\left(\fint_{\mathcal{Q}_{\varrho}}
\left|\mathcal{G}_{1}h(s,z)-\left(\mathcal{G}_{1}h\right)_{\mathcal{Q}_{\varrho}}\right|^{p}\,dzds\right)^{\frac{1}{p}}\\
&\lesssim\left(\mathcal{M}_{t,x}\left|h\right|^{p}(t,x)\right)^{\frac{1}{p}}
+\left(\mathcal{M}_{t}\mathcal{M}_{x}\left|h\right|^{p}(t,x)\right)^{\frac{1}{p}},
\end{align*}
which establishes \eqref{supp estimate 5.4}.
The estimates for $\mathcal{G}_{0}$ in \eqref{supp estimate 5.1} and \eqref{supp estimate 5.3} follow similarly, with the constants now additionally depending on $T$ due to the time truncation in the definition of $\mathcal{G}_{0}$.
\end{proof}
\section{Main Result}
We establish the main theorem of this paper.

\begin{theorem}\label{Theorem 1}
For $w_{0}\equiv 0$, $\gamma\in\mathbb{R}$ let $0<\alpha<1$, $0<T<\infty$, $1<p,q<\infty$, $\mu_{1}(\cdot)\in A_{p}(\mathbb{R}^{d})$, $\mu_{2}(\cdot)\in A_{q}(\mathbb{R})$ and $h\in\mathcal{H}^{\phi,\gamma}_{p,q}\left(\mu_{1},\mu_{2},T\right)$. Then the time-space fractional equation (TSFEs) \eqref{TSFE} has a unique solution $w$ in $\mathcal{H}^{\alpha,\phi,\gamma+2}_{p,q,0}\left(\mu_{1},\mu_{2},T\right)$, satisfying:
\begin{align*}
\left\|\phi(\Delta)w\right\|_{\mathcal{H}^{\phi,\gamma}_{p,q}\left(\mu_{1},\mu_{2},T\right)}&\leq C_{0}\left\|h\right\|_{\mathcal{H}^{\phi,\gamma}_{p,q}\left(\mu_{1},\mu_{2},T\right)},\\
\left\|\partial^{\alpha}_{t}w\right\|_{\mathcal{H}^{\phi,\gamma}_{p,q}\left(\mu_{1},\mu_{2},T\right)}
+\left\|\phi(\Delta)w\right\|_{\mathcal{H}^{\phi,\gamma}_{p,q}\left(\mu_{1},\mu_{2},T\right)}&\leq C_{1}\left\|h\right\|_{\mathcal{H}^{\phi,\gamma}_{p,q}\left(\mu_{1},\mu_{2},T\right)},
\end{align*}
where $C_{0}=C_{0}\left(\alpha,p,q,d,[\mu_{1}]_{p},[\mu_{2}]_{q}\right)$ and $C_{1}=C_{1}\left(\alpha,p,q,d,[\mu_{1}]_{p},[\mu_{2}]_{q},T\right)$.
\end{theorem}

\begin{proof}
We only to verify $\gamma=0$.

First, we prove the a priori estimates and uniqueness.

Let $w\in\mathcal{H}^{\alpha,\phi,2}_{p,q,0}\left(\mu_{1},\mu_{2},T\right)$ be a solution to TSFE \eqref{TSFE}. There exists a sequence $w_{n}(t,x)\in C_{c}^{\infty}\left((0,\infty)\times\mathbb{R}^{d}\right)$ such that
\begin{align*}
w_{n}\rightarrow w\text{ in }\mathcal{H}^{\alpha,\phi,2}_{p,q,0}\left(\mu_{1},\mu_{2},T\right).
\end{align*}
Let $h_{n}(t,x)=\partial^{\alpha}_{t}w_{n}(t,x)-\phi(\Delta)w_{n}(t,x)$. Then, by Lemma \ref{Solution Resp}, we have
\begin{align*}
w_{n}(t,x)=\int_{0}^{t}\int_{\mathbb{R}^{d}}\mathcal{S}_{\alpha,1,\phi}(t-\tau,x-z)h_{n}(\tau,z)\,dzd\tau,\quad 0<\tau<t<T.
\end{align*}
This implies $w_{n}=\mathcal{G}_{0}h_{n}$ and $\phi(\Delta)w_{n}=\mathcal{G}_{1}h_{n}$ by extending $h(\tau,x)=0$ for $\tau<0$.

Let $\chi_{k}\in C^{\infty}(\mathbb{R})$ satisfy $0\leq\chi_{k}\leq 1$, $\chi_{k}=1$ for $t\leq T$, and $\chi_{k}=0$ for $t\geq T+1/k$.

Given $\mu_1 \in A_p(\mathbb{R}^d)$ and $\mu_2 \in A_q(\mathbb{R})$, Remark~\ref{Ap weight proposition} yields exponents $p^* \in (1,p)$ and $q^* \in (1,q)$ with $\mu_1 \in A_{p^*}$, $\mu_2 \in A_{q^*}$. We select $p_0 > 1$ satisfying
\begin{align*}
p^* < \frac{p}{p_0} < p, \quad q^* < \frac{q}{q_0} < q,
\end{align*}
which implies $\mu_{1}\in A_{\frac{p}{p_{0}}}$ and $\mu_{2}\in A_{\frac{q}{q_{0}}}$.

For $i=0,1$, by using the \emph{Fefferman-Stein} theorem (see \cite{Grafakos}) , Lemma \ref{supp estimate 5} and a version of the Hardy-Littlewood theorem \cite[Corollary 2.6]{Dong 3}, we derive:
\begin{align*}
&\left\|\mathcal{G}_{i}\left(h_{n}\chi_{k}\right)\right\|_{\mathcal{L}_{p,q}\left(\mu_{1},\mu_{2}\right)}\\
&\lesssim
\left\|\left(\mathcal{G}_{i}\left(h_{n}\chi_{k}\right)\right)^{\sharp}\right\|_{\mathcal{L}_{p,q}\left(\mu_{1},\mu_{2}\right)}\\
&\leq C_{i}\left\|\left(\mathcal{M}_{t,x}\left|h_{n}\chi_{k}\right|^{p_{0}}\right)^{\frac{1}{p_{0}}}
+\left(\mathcal{M}_{t}\mathcal{M}_{x}\left|h_{n}\chi_{k}\right|^{p_{0}}\right)^{\frac{1}{p_{0}}}\right\|_{\mathcal{L}_{p,q}
\left(\mu_{1},\mu_{2}\right)}\\
&\leq C_{i}\left\|\mathcal{M}_{t,x}\left|h_{n}\chi_{k}\right|^{p_{0}}\right\|^{\frac{1}{p_{0}}}
_{L_{q/p_{0}}\left((0,\infty),\mu_{2}dt;L_{p/p_{0}}(\mu_{1})\right)}
+C_{i}\left\|\mathcal{M}_{t}\mathcal{M}_{x}\left|h_{n}\chi_{k}\right|^{p_{0}}\right\|^{\frac{1}{p_{0}}}
_{L_{q/p_{0}}\left((0,\infty),\mu_{2}dt;L_{p/p_{0}}(\mu_{1})\right)}.
\end{align*}
Applying the weighted \emph{Hardy-Littlewood maximal} function (see \cite{Grafakos}) to $x$ and $t$, we obtain:
\begin{align*}
\left\|\mathcal{G}_{i}\left(h_{n}\chi_{k}\right)\right\|_{\mathcal{L}_{p,q}\left(\mu_{1},\mu_{2}\right)}\leq C_{i}\left\|h_{n}\chi_{k}\right\|_{\mathcal{L}_{p,q}\left(\mu_{1},\mu_{2}\right)},
\end{align*}
where $C_{0}=C_{0}(\alpha,p,q,d,[\mu_{1}]_{p},[\mu_{2}]_{q},T)$ and $C_{1}=C_{1}(\alpha,p,q,d,[\mu_{1}]_{p},[\mu_{2}]_{q})$.

Thus, we have:
\begin{align*}
\left\|\mathcal{G}_{i}h_{n}\right\|_{\mathcal{L}_{p,q}\left(\mu_{1},\mu_{2},T\right)}
=\left\|\mathcal{G}_{i}\left(h_{n}\chi_{k}\right)\right\|_{\mathcal{L}_{p,q}\left(\mu_{1},\mu_{2}\right)}\leq C_{i}\left\|h_{n}\chi_{k}\right\|_{\mathcal{L}_{p,q}\left(\mu_{1},\mu_{2}\right)}.
\end{align*}
Taking $k\rightarrow\infty$, we conclude:
\begin{align}\label{prior estimate}
\left\|\mathcal{G}_{i}h_{n}\right\|_{\mathcal{L}_{p,q}\left(\mu_{1},\mu_{2},T\right)}\leq C_{i}\left\|h_{n}\right\|_{\mathcal{L}_{p,q}\left(\mu_{1},\mu_{2},T\right)}.
\end{align}

Note that $w_{n}=\mathcal{G}_{0}h_{n}$, $\phi(\Delta)w_{n}=\mathcal{G}_{1}h_{n}$, and $\partial^{\alpha}_{t}w_{n}=\phi(\Delta)w_{n}+h_{n}$. Therefore,
\begin{align*}
\left\|w\right\|_{\mathcal{H}^{\alpha,\phi,2}_{p,q}(\mu_{1},\mu_{2},T)}&=\lim_{n\rightarrow\infty}
\left\|w_{n}\right\|_{\mathcal{H}^{\alpha,\phi,2}_{p,q}(\mu_{1},\mu_{2},T)}\\
&=\lim_{n\rightarrow\infty}\left(\left\|\partial^{\alpha}_{t}w_{n}\right\|_{\mathcal{L}_{p,q}(\mu_{1},\mu_{2},T)}
+\left\|\phi(\Delta)w_{n}\right\|_{\mathcal{L}_{p,q}(\mu_{1},\mu_{2},T)}\right)\\
&\leq C\left\|h\right\|_{\mathcal{L}_{p,q}(\mu_{1},\mu_{2},T)},
\end{align*}
where  $C=\left(\alpha,p,q,d,[\mu_{1}]_{p},[\mu_{2}]_{q},T\right)$.

Finally, we verify the existence of the solution. Since $h\in \mathcal{L}_{p,q}(\mu_{1},\mu_{2},T)$, there exists a sequence $h_{n}\in C^{\infty}_{c}((0,\infty)\times\mathbb{R}^{d})$ such that $h_{n}\rightarrow h$ in $\mathcal{L}_{p,q}(\mu_{1},\mu_{2},T)$. Define
\begin{align*}
w_{n}(t,x)=\int_{0}^{t}\int_{\mathbb{R}^{d}}\mathcal{S}_{\alpha,1,\phi}(t-\tau,x-y)h_{n}(\tau,y)\,dyd\tau.
\end{align*}
Then $w_{n}$ satisfies TSFEs \eqref{TSFE} for any $n\in\mathbb{N}^{+}$. By the a priori estimate \eqref{prior estimate}, for any $m,k$, we have
\begin{align*}
\left\|w_{m}-w_{k}\right\|_{\mathcal{H}^{\alpha,\phi,2}_{p,q}}\leq C\left\|h_{m}-h_{k}\right\|_{\mathcal{L}_{p,q}\left(\mu_{1},\mu_{2},T\right)}.
\end{align*}
This shows that $\left\{w_{n}\right\}$ is a Cauchy sequence. Taking $n\rightarrow\infty$, we obtain the limit $w\in\mathcal{H}^{\alpha,\phi,2}_{p,q}$, which satisfies TSFEs \eqref{TSFE}. The proof is complete.
\end{proof}

\begin{remark}
It is clear that Theorem~\ref{Theorem 1} generalizes the result of Han \cite{Han}, which only deals with the case $\phi(\Delta)=\Delta$.
\end{remark}

Next, we consider the case of non-zero initial data. For $0<T<\infty$, consider a weight function $\mu_{2}^{\circ}\in A_{q}\left(0,T\right)$, i.e., $\mu_{2}^{\circ}$ is locally integrable and satisfies
\begin{align*}
\sup_{0\leq b<c\leq T}\left(\frac{1}{c-b}\int_{b}^{c}\mu_{2}^{\circ}(t)\,dt\right)\left(\frac{1}{c-b}
\int_{b}^{c}\mu_{2}^{\circ}(t)^{-\frac{1}{q-1}}\,dt\right)^{q-1}<\infty.
\end{align*}
Note that $k^{\circ}(t):=t^{-\alpha}$ is right continuous and decreasing on $\left(0,T\right)$, and there exists $\varepsilon>0$ such that
\[
\lambda^{-1+\varepsilon}\lesssim\frac{k^{\circ}(\lambda t)}{k^{\circ}(t)}\lesssim \lambda^{-\varepsilon},\quad \text{for any }0<t\leq \lambda t<T.
\]
If we assume that the weight function $\mu_{2}^{\circ}(t)$ satisfies
\begin{align*}
\lambda^{\varepsilon'}\lesssim\frac{W^{\circ}(k^{\circ\star}(\lambda t))}{W^{\circ}(k^{\circ\star}(t))}\lesssim\lambda^{p-\varepsilon'},\quad \text{for any }0<t\leq \lambda t<\frac{1}{k(T)},
\end{align*}
where $W^{\circ}(t)=\int_{0}^{t}\mu_{2}^{\circ}(s)\,ds$ and $k^{\circ\star}(t)=\left(k^{\circ}\mathbb{I}_{(0,T)}\right)^{-1}(\frac{1}{t})$, then by \cite{Choi J}, there exist an extension $\mu_{2}\in A_{q}\left(\mathbb{R}\right)$ of $\mu_{2}^{\circ}$ and an extension $k(t)$ of $k^{\circ}(t)$ (in fact, here $k(t)=t^{-\alpha}$) such that $W\circ k^{\star}\in I_{\circ}(0,q)$, where $W(t)=\int_{0}^{t}\mu_{2}(s)\,ds$ and $I_{\circ}(0,q)$ is the function class defined by
\begin{align*}
I_{\circ}\left(0,q\right)=\left\{f:\sup_{t>0}\frac{f(\lambda t)}{f(t)}=\circ(1) \text{ as }\lambda\rightarrow 0,\;
\sup_{t>0}\frac{f(\lambda t)}{f(t)}=\circ(\lambda^{q}) \text{ as }\lambda\rightarrow \infty\right\}.
\end{align*}
Thus we obtain $\left(W\circ k^{\star}\right)^{\frac{1}{q}}\in I_{\circ}(0,1)$. Consider the following initial value space $N_{\alpha,p,\phi}$:
\[
\left\|w_{0}\right\|_{N_{\alpha,p,\phi}}=
\begin{cases}
\left\|w_{0}\right\|_{B^{\phi,\left(W\circ k^{\star}\right)^{\frac{1}{q}}(\gamma+2,\gamma)}_{p,q}(\mu_{1})}, & \text{if } \left(W\circ k^{\star}\right)^{\frac{1}{q}}(2^{-2})>2^{-2}, \\
\left\|w_{0}\right\|_{H^{\phi,\gamma}_{p}(\mu_{1})}, & \text{if } \left(W\circ k^{\star}\right)^{\frac{1}{q}}(2^{-2})\leq2^{-2},
\end{cases}
\]
where $B^{\phi,\left(W\circ k^{\star}\right)^{\frac{1}{q}}(\gamma+2,\gamma)}_{p,q}(\mu_{1})$ is the weighted $\phi$-type Besov space, i.e.,
\begin{align*}
\left\|w\right\|_{B^{\phi,\left(W\circ k^{\star}\right)^{\frac{1}{q}}(\gamma+2,\gamma)}_{p,q}(\mu_{1})}
=\left\|\chi^{\phi}w\right\|_{L_{p}(\mu_{1})}+\left\|\left\{\left((W\circ k^{\star})^{\frac{1}{q}}(2^{-2})\right)^{j}2^{(\gamma+2)j}\left\|\bigtriangleup_{j}^{\phi}w\right\|_{L_{p}(\mu_{1})}\right\}_{j\in \mathbb{N}_{+}}\right\|_{l^{q}},
\end{align*}
where $\chi^{\phi}=1-\sum_{j\geq 0}\bigtriangleup_{j}^{\phi}$, and $\bigtriangleup_{j}^{\phi}:=\mathcal{F}^{-1}\left(\psi(2^{-j}\phi(\left|\xi\right|^{2}))\right)$, with $\psi$ being a bump function supported in $\left\{\xi:\frac{1}{2}\leq \left|\xi\right|\leq 2\right\}$.

For convenience, in the following we use $\mu_{2}$ to denote $\mu^{\circ}_{2}$, and assume that $\left(W\circ k^{\star}\right)^{\frac{1}{q}}(2^{-2})>2^{-2}$.

\begin{theorem}\label{Theorem 2}
For $\gamma\in\mathbb{R}$, let $0<\alpha<1$, $0<T<\infty$, $1<p,q<\infty$, $\mu_{1}(\cdot)\in A_{p}(\mathbb{R}^{d})$, $\mu_{2}(\cdot)\in A_{q}(\mathbb{R})$, and $h\in\mathcal{H}^{\phi,\gamma}_{p,q}\left(\mu_{1},\mu_{2},T\right)$. Then for any $w_{0}\in B^{\phi,\left(W\circ k^{\star}\right)^{\frac{1}{q}}(\gamma+2,\gamma)}_{p,q}(\mu_{1}) $, the time-space fractional equation (TSFEs) \eqref{TSFE} has a unique solution $w$ in $\mathcal{H}^{\alpha,\phi,\gamma+2}_{p,q}\left(\mu_{1},\mu_{2},T\right)$, satisfying:
\begin{align}\label{estimate initial value}
\left\|w\right\|_{\mathcal{H}^{\alpha,\phi,\gamma}_{p,q}\left(\mu_{1},\mu_{2},T\right)}
+\left\|\phi(\Delta)w\right\|_{\mathcal{H}^{\phi,\gamma}_{p,q}
\left(\mu_{1},\mu_{2},T\right)}\lesssim_{C}\left\|h\right\|_{\mathcal{H}^{\phi,\gamma}_{p,q}
\left(\mu_{1},\mu_{2},T\right)}+\left\|w_{0}\right\|_{B^{\phi,\left(W\circ k^{\star}\right)^{\frac{1}{q}}(\gamma+2,\gamma)}_{p,q}(\mu_{1}) },
\end{align}
where the constant depends on $\alpha,p,q,\mu_{1},\mu_{2},\delta,T$.
\end{theorem}

\begin{proof}
It suffices to consider the case $\gamma=0$.

Note that $\left(W\circ k^{\star}\right)^{\frac{1}{q}}\in I_{\circ}(0,1)$. Combined with \cite[Proposition A.3]{Choi J}, we obtain
\begin{align*}
\left(H^{\phi,2}_{p}(\mu_{1}),H^{\phi,0}_{p}(\mu_{1})\right)_{\left(W\circ k^{\star}\right)^{\frac{1}{q}},q}
=B^{\phi,\left(W\circ k^{\star}\right)^{\frac{1}{q}}(2,0)}_{p,q}(\mu_{1}).
\end{align*}
By applying \cite[Theorem 5.3]{Choi J}, we obtain
\begin{align*}
\left\|w(0,\cdot)\right\|_{B^{\phi,\left(W\circ k^{\star}\right)^{\frac{1}{q}}(2,0)}_{p,q}(\mu_{1})}&\lesssim_{C}\left\|w\right\|_{L_{q}((0,T),\mu_{2}\,dt;H^{\phi,2}_{p}(\mu_{1}))}
+\left\|h\right\|_{L_{q}\left((0,T),\mu_{2}\,dt;L_{p}(\mu_{1})\right)}\\
&\lesssim_{C}\left\|w\right\|_{\mathcal{H}^{\phi,2}_{p,q}\left(\mu_{1},\mu_{2},T\right)}
+\left\|\partial^{\alpha}_{t}w\right\|_{\mathcal{H}^{\phi,0}_{p,q}\left(\mu_{1},\mu_{2},T\right)}.
\end{align*}
On the other hand, for any $w_{0}\in B^{\phi,\left(W\circ k^{\star}\right)^{\frac{1}{q}}(2,0)}_{p,q}(\mu_{1})$, by applying \cite[Theorem 1.6]{Choi J}, we obtain that there exist $w\in L_{q}(\mathbb{R}_{+},\mu_{2}\,dt;H^{\phi,2}_{p}(\mu_{1}))$ and $h\in L_{q}(\mathbb{R}_{+},\mu_{2}\,dt;L_{p}(\mu_{1}))$ such that $\partial^{\alpha}_{t}w=h$ and
\begin{align*}
\left\|w\right\|_{L_{q}(\mathbb{R}_{+},\mu_{2}\,dt;H^{\phi,2}_{p}(\mu_{1}))}
+\left\|h\right\|_{L_{q}(\mathbb{R}_{+},\mu_{2}\,dt;L_{p}(\mu_{1}))}\lesssim_{C}\left\|w_{0}\right\|_{B^{\phi,\left(W\circ k^{\star}\right)^{\frac{1}{q}}(2,0)}_{p,q}(\mu_{1})}.
\end{align*}
Therefore, we have
\begin{align*}
\left\|w\right\|_{\mathcal{H}^{\phi,2}_{p,q}\left(\mu_{1},\mu_{2},T\right)}
+\left\|\partial^{\alpha}_{t}w\right\|_{\mathcal{H}^{\phi,0}_{p,q}\left(\mu_{1},\mu_{2},T\right)}&\lesssim \left\|w\right\|_{L_{q}(\mathbb{R}_{+},\mu_{2}\,dt;H^{\phi,2}_{p}(\mu_{1}))}
+\left\|h\right\|_{L_{q}(\mathbb{R}_{+},\mu_{2}\,dt;L_{p}(\mu_{1}))}\\
&\quad+\left\|w_{0}\right\|_{B^{\phi,\left(W\circ k^{\star}\right)^{\frac{1}{q}}(2,0)}_{p,q}(\mu_{1})}.
\end{align*}
Hence, we obtain that there exists $w_{1}\in \mathcal{H}^{\alpha,\phi,2}_{p,q}\left(\mu_{1},\mu_{2},T\right)$ satisfying
\begin{align*}
\partial^{\alpha}_{t}w_{1}=f,\quad w_{1}(0,\cdot)=w_{0},
\end{align*}
and
\[
\left\|\partial^{\alpha}_{t}w_{1}\right\|_{\mathcal{H}^{\phi,0}_{p,q}\left(\mu_{1},\mu_{2},T\right)}
+\left\|w_{1}\right\|_{\mathcal{H}^{\phi,2}_{p,q}\left(\mu_{1},\mu_{2},T\right)}\lesssim\left\|w_{0}\right\|_{B^{\phi,\left(W\circ k^{\star}\right)^{\frac{1}{q}}(2,0)}_{p,q}(\mu_{1})}.
\]
By using Theorem \ref{Theorem 1}, there exists $w_{2}\in \mathcal{H}^{\alpha,\phi,2}_{p,q,0}\left(\mu_{1},\mu_{2},T\right)$ satisfying
\begin{align*}
\partial^{\alpha}_{t}w_{2}=\phi(\Delta)w_{2}+h-f+\phi(\Delta)w_{1},\quad w_{2}(0)=0.
\end{align*}
Therefore, $w=w_{1}+w_{2}\in\mathcal{H}^{\alpha,\phi,2}_{p,q}\left(\mu_{1},\mu_{2},T\right)$, $w(0)=w_{0}$, and it satisfies \eqref{estimate initial value}. Moreover, the uniqueness follows directly from Theorem \ref{Theorem 1}.
\end{proof}

\begin{remark}
For $\mu_{2}\equiv 1$, it is easy to see that $k^{\star}(t)=k^{-1}(t^{-1})=\left(\Gamma(1-\alpha)t\right)^{\frac{1}{\alpha}}$ and $\left(W\circ k^{\star}\right)^{\frac{1}{q}}=\left(\Gamma(1-\alpha)t\right)^{\frac{1}{\alpha q}}$. Hence, when $\alpha q>1$,
\[
\left\|w_{0}\right\|_{N_{\alpha,p,\phi}}\sim \left\|w_{0}\right\|_{B^{\phi,\left(W\circ k^{\star}\right)^{\frac{1}{q}}(\gamma+2,\gamma)}_{p,q}(\mu_{1})}\sim\left\|w_{0}\right\|_{B^{\phi,\gamma+2-\frac{2}{\alpha q}}_{p,q}(\mu_{1})}.
\]
Thus, Theorem~\rm\ref{Theorem 2} extends the result of Kim \cite[Theorem 2.8]{Kim1}.
\end{remark}

\noindent{\bf Declaration of competing interest}\\
The authors declare that they have no competing interests.\\
\noindent{\bf Data availability}\\
No data was used for the research described in the article.\\
\noindent{\bf Acknowledgements}\\
This work was supported by National Natural Science Foundation of China (12471172), Fundo para o Desenvolvimento das Ci\^{e}ncias e da Tecnologia of Macau (No. 0092/2022/A) and Hunan Province Doctoral Research Project CX20230633.

\end{document}